\documentclass[journal]{IEEEtran}
              \usepackage{hyperref}                                            
 \usepackage{amssymb,amsmath,mathrsfs,graphicx,subfigure,multicol,setspace,stmaryrd,amsthm}

\usepackage[T1]{fontenc}

\usepackage{tikz}
\usetikzlibrary{arrows}

\newcommand{\lra}{\longrightarrow}
\newcommand{\Rnn}{\bar{\mathbb R}_+^n}

\newcommand{\diag}{\mathop{\mbox{diag}}}
\newcommand{\sgn}{\mathop{\mbox{sgn}}}
\newcommand{\supp}{\mathop{\mbox{supp}}}
\newcommand{\im}{\mathop{\mbox{Im}}}
\newcommand{\tV}{{\tilde{V}}}
\newcommand{\cW}{{\mathcal W}}
\newcommand{\co}{\mathop{\mbox{co}}}
  \newtheorem{theorem}{Theorem}
  \newtheorem{lemma}{Lemma}
  \newtheorem{corollary}[theorem]{Corollary}
  \newtheorem{proposition}[theorem]{Proposition}
  \newtheorem{definition}{Definition}
  \newtheorem{remark}{Remark}

  \newcommand{\R}{{\rm \bf R}}
 
\title{New Approach to the Stability of Chemical Reaction Networks: Piecewise Linear in Rates Lyapunov Functions}

\author{M. Ali Al-Radhawi, \emph{Student Member, IEEE},\thanks{This work was partially supported by Leverhulme Trust under award titled ``Structural conditions for oscillation in chemical reaction networks''. \newline M. Ali Al-Radhawi is with the Department of Electrical and Electronic Engineering, Imperial College London, London SW7 2AZ, United Kingdom. Email:\texttt{ m.rizvi11@imperial.ac.uk}} and  David Angeli, \emph{Fellow, IEEE} 
\thanks{David Angeli is with the Department of Electrical and Electronic Engineering, Imperial College London, London SW7 2AZ, United Kingdom. Also, he is with Dipartimento di Ingegneria dell'Informazione, University of Florence, 50139 Florence, Italy. Email:\texttt{ d.angeli@imperial.ac.uk}}}

\begin{document}

\maketitle

\begin{abstract}
Piecewise-Linear in Rates (PWLR) Lyapunov functions are introduced for a class of Chemical Reaction Networks (CRNs). In addition to their simple structure, these functions are robust with respect to arbitrary monotone reaction rates, of which Mass-Action is a special case. The existence of such functions ensures the convergence of trajectories towards equilibria, and can be used to establish their asymptotic stability with respect to the corresponding stoichiometric compatibility class. We give the definition of these Lyapunov functions, prove their basic properties, and provide algorithms for constructing them. Examples are provided, relationship with consensus dynamics are discussed, and future directions are elaborated.
\end{abstract}


\section{Introduction}

The study of the dynamic behavior of Chemical (or complex) Reaction
Networks (CRNs) finds its roots in Boltzmann's $H$-theorem \cite{tolman38} and the subsequent chemical engineering literature \cite{horn72,clarke80,feinberg95}. Recently, this area has sparked a growing interest in the control and systems community \cite{bastin99,sontag01,haddad09,angeli09tut}. This is especially in the light
of the challenges posed by  the emerging field of molecular systems
biology. In this respect, one of the main goals is to understand the cell behavior and
function at the level of chemical interactions and, in particular, the characterization of qualitative
features of dynamical behavior (stability, periodic orbits, chaos, etc.) resulting from such interactions.

However, a major difficulty in this field is the very large degree of uncertainty inherent
in the models of cellular biochemical networks. 
Thus, it is imperative to develop tools that are ``robust'' in the sense of being able to provide
useful conclusions based only upon information regarding the qualitative features of the network,
and not the precise values of parameters or even the specific form of reaction kinetics. Of course, this goal is
often unachievable, since the dynamical behavior may be subject to bifurcation
phenomena which are critically dependent on parameter values.

Nevertheless, research by many \cite{horn72,feinberg95}, \cite{clarke80}, \cite{sontag01}, \cite{angeli08,angeli10} has resulted in the identification of classes of chemical reaction networks for which important dynamical properties such as stability, monotonicity, persistence, etc can be checked based on structural information only regardless of the parameters involved. In this work, we follow this line of research by investigating stability properties for a wide class of chemical reaction networks.

Earlier work regarding asymptotic stability has concentrated on the concepts of detailed and complex balancing \cite{horn72}. It was shown afterwards that Mass-Action networks, which satisfy the graphical condition of being weakly reversible with deficiency zero, are complex-balanced \cite{feinberg95}. Therefore, there exists a unique equilibrium in the interior of each class which is locally asymptotically stable for complex-balanced networks. This theorem is remarkable since asymptotic stability was established independently of the kinetic constants involved.
It was shown later that if there are no equilibria on the boundary of the
class, then global asymptotic stability of the interior equilibrium holds \cite{sontag01}. However, the question of global asymptotic stability for complex-balanced CRNs remains open in general.

 The subclass of unimolecular CRNs that have a compartmental matrix can be shown to be stable \cite{maeda78}. Another approach is based on the notion of monotone systems \cite{angeli10}. Once monotonicity is established, convergence theorems for monotone systems can be applied.
\begin{figure}[t]
\centering
\scalebox{0.77}{
\centering
\begin{tikzpicture}[node distance=1.5cm,>=stealth',bend angle=45,auto]
  \tikzstyle{S}=[circle,thick,draw=black,fill=black,minimum size=2.5mm]
  \tikzstyle{R}=[rectangle,very thin,draw=black,
  			  fill=black,minimum width=.01in, minimum height=9mm]
\begin{scope}
 \node [S] (X2) [label=above:$X_2$] {};
   \node [S] (X4) [below of=X2,xshift=1.5cm,label=above:$X_3$]   {};

    \node [R] (R3) [left of=X2,yshift=-0.75cm] {};

             \node [R] (R2) [right of=X2] {}
             edge [<-]   node[above,swap] {$2$}                (X2);
             \node [S] (X3) [right of=R2, label=above:$X_4$] {}
            edge [<-]                  (R2);
                 \node [R] (R1) [right of=X3,yshift=-0.75cm] {};

                  \node [S] (X1) [below of=R2, yshift=-1.2cm, label=above:$X_1$] {};

     \draw[->] (R1) .. controls +(30:4.2cm) and +(135:3.0cm) .. (X2);
     \draw[->] (R3) .. controls +(60:0.75cm) and +(200:0.75cm) .. (X2);
     \draw[->] (R3) .. controls +(-60:0.75cm) and +(-180:1.75cm) .. (X4);
     \draw[->] (X4) .. controls +(0:2cm) and +(220:0.55cm) .. (R1);
     \draw[->] (X3) .. controls +(0:0.75cm) and +(140:0.55cm) .. (R1);
      \draw[->] (R1) .. controls +(-20:2.4cm) and +(0:2.4cm) .. (X1);
      \draw[->] (X1) .. controls +(-180:3.4cm) and +(-180:2.0cm) .. (R3);

   \end{scope}
\end{tikzpicture}}
\caption{Bipartite graph representation of the example CRN \eqref{crn_intro}.}\label{f.crn_intro}
\end{figure}
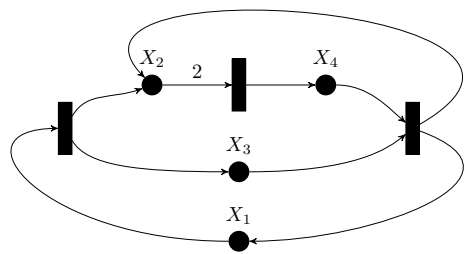

Despite the wide range of results available, they only cover a small subset of CRNs which are in practice observed to be stable. For instance, consider the following CRN which is depicted in Figure \ref{f.crn_intro} :
 \begin{align}\nonumber
     X_1 &\lra X_2+X_3,\\ \label{crn_intro}
    2X_2& \lra X_4, \\ \nonumber
    X_3 + X_4 &\lra X_1+X_2.
 \end{align}
The corresponding ODE system is:
 \begin{align}\label{e.crn_intro}
 \dot x =  \begin{bmatrix} -1 & 0& 1 \\ 1 & -2 & 1 \\ 1 & 0 & -1 \\ 0 & 1 & -1 \end{bmatrix}\begin{bmatrix} R_1(x_1)\\  R_2(x_2) \\ R_3(x_3,x_4) \end{bmatrix},
 \end{align}
 $x \in \mathbb R^4$, and $R_1,R_2,R_3$ are the reaction rate functions, where the species appearing in the argument correspond to the reactants.

Despite its simplicity, its stability can not be established via previous results in the literature \cite{feinberg95,angeli10,sontag01,maeda78} even for the Mass-Action case. However, consider the following:
\begin{align}&V(x)= \max\{ |R_1(x) - R_2(x)|,|R_2(x) - R_3(x)|, \nonumber \\ \label{e.lyap_eg} & \hspace{1.83in}  |R_3(x) - R_1(x)| \}.\vspace{0em} \end{align}
It can be verified that $V$ is decreasing along all trajectories for \emph{any choice} of monotone reaction rate functions, where exact conditions are to be detailed in \S 2. Using this, asymptotic stability of the equilibrium set can be established. In fact, in this case global asymptotic stability can be shown.

In this work, we consider the problem of stability of CRNs by invoking Lyapunov functions of the form \eqref{e.lyap_eg}, which we call \emph{Piecewise Linear in Rates (PWLR) Lyapunov functions}. Not only they have simple structure, these functions are robust with respect to arbitrary positive the kinetic constants and require mild assumptions on the reaction kinetics. Mass-Action kinetics is a special case of the admissible kinetics. \\
The concept of utilizing convex piecewise linear functions as Lyapunov functions to establish the stability is not a new one. For instance, it has been used for special nonlinear systems \cite{rosenbrock63}, linear systems \cite{willems76,molchanov89}, and consensus dynamics \cite{moreau04}. Maeda et al. \cite{maeda78} used a piecewise linear function in term of the time derivative of the states.  \\ In this work, we extend this approach to classes of nonlinear systems which have a graphical structure. By identifying nodes which are represented by nonlinear functions, we can construct Lyapunov functions which are piecewise linear in terms of the node functions. This approach is used for reaction networks where node functions are reaction rates.  The existence of such PWLR Lyapunov functions ensures the convergence of trajectories toward the equilibria. They can be used to establish the asymptotic stability of the equilibria with respect to their stoichiometric class.

The paper is organized as follows. In Section 2, we present the main definitions and assumptions. Section 3 includes the definition of PWLR Lyapunov function, and the algorithms for checking candidate functions, while section 4 presents their properties. Various constructions of PWLR functions are introduced in Section 5. In Section 6, we present some illustrative examples, and Section 7 contains the conclusion. The proofs are presented in the Appendix.

\paragraph*{Notation} Let $A \subset \mathbb R^n$ be a set, then $A^\circ, \bar A, \partial A, \co A$ denote its interior, closure, boundary, and convex hull, respectively. The tangent cone to $A$ at $x \in A$ is denoted by $T_x A$. Let $x \in \mathbb R^n$ be a vector, then its $\ell_\infty$-norm is $\|x\|_{\infty} = \max_{1\le i \le n} |x_i|$, and the $\ell_1$-norm is $\|x\|_1=\sum_{i=1}^n |x_i|$. The support of $x$ is defined as $\supp(x) = \{ i \in \{1,.., n\} | x_i \ne 0 \}$. The inequalities $x\ge0,\, x>0,\, x\gg 0$ denote elementwise nonnegativity, elementwise nonnegativity with at least one positive element, and elementwise positivity, respectively. Let $A \in \mathbb R^{n \times \nu}$, then $\ker(A)$ denotes the kernel or null-space of $A$, while $\im (A)$ denotes the image space of $A$. The all-ones vector is denoted by $\mathbf 1$, where its dimension can be inferred from the context. Let $V:D \to \mathbb R$, then the kernel of $V$ is $\ker(V)=V^{-1}(0)$.

 \section{Background on Reaction Networks}
  The field of CRN dynamics has an established literature \cite{horn72,clarke80,feinberg95,angeli09tut}. In this section, we review the relevant notations and definitions.
\subsection{Ordinary Differential Equations Formulation}
A Complex (or Chemical) Reaction Network (CRN) is defined by a set of species $\mathscr S=\{X_1,..,X_n\}$, and a set of reactions $\mathscr R=\{\R_1,...,\R_\nu\}$. Each reaction is denoted as:
\begin{equation}\label{e.reaction}
    \R_j: \quad \sum_{i=1}^n \alpha_{ij} X_i \longrightarrow \sum_{i=1}^n \beta_{ij} X_i, \ j=1,..,\nu,
\end{equation}
where $\alpha_{ij}, \beta_{ij}$ are nonnegative integers called \emph{stoichiometry coefficients}. The expression on the left-hand side is called the \emph{reactant complex}, while the one on the right-hand side is called the \emph{product complex.} The forward arrow refers to the idea that the transformation of reactants into products is only occurring in the direction of the arrow. If the transformation is occurring also in the opposite direction, the reaction is said to be \emph{reversible} and its reverse is listed as a separate reaction.   For convenience, the reverse reaction of $\R_j$ is denoted as $\R_{-j}$.
Note that we allow reactant or product complex to be empty, though not simultaneously. This is used to model external inflows and outflows of the CRN.

A nonnegative concentration $x_i$ is associated to each species $X_i$. Each chemical reaction $\R_j$ takes place continuously in time at a rate $R_j:\Rnn \to \bar{\mathbb R}_+$. We assume that the reaction rate satisfies the following:
 \begin{enumerate}
   \item[\bf A1.] it is a  $\mathscr C^1$ function, i.e. continuously differentiable;
   \item[\bf A2.] $x_i =0 \Rightarrow R_j(x)=0$, for all $i$ and $j$ such that $\alpha_{ij} > 0$; 
   \item[\bf A3.] it is nondecreasing with respect to its reactants, i.e
\begin{equation}\label{e.rate_mono}
    \frac{\partial R_j}{\partial x_i}(x)  \left \{\begin{array}{ll} \geq 0 & : \alpha_{ij} > 0 \\  =0 &: \alpha_{ij}=0 \end{array} \right . .
\end{equation}
\item[\bf A4.] The inequality in \eqref{e.rate_mono} holds strictly for all $x \in \mathbb R_+^n$.
 \end{enumerate}
\begin{remark}
Although there exist special biochemical models in which monotonicity does not apply \cite{nonmonotone}, this assumption conforms to the mostly used and popular reaction rate models including Mass-Action, Michaelis-Menten, and Hill kinetics. Furthermore, monotonic dependence of the reaction rate on the concentration of its reactants captures the basic intuition about the nature of a reaction since, as the concentration of reactants increases, the likelihood of collision between molecules increases, and hence the rate of the reaction.
\end{remark}
\begin{remark}\label{rem.sign} It can be noted that if the Jacobian matrix $\frac{\partial R^*}{\partial x}$ satisfies \eqref{e.rate_mono}, then A3 is also satisfied for all Jacobian matrices that have the same sign pattern, i.e. the same sign class.
\end{remark}

Considering a microscopic setup and the associated statistical thermodynamics considerations, the following widely-used expression for the reaction rate function can be derived:
\begin{equation}\label{e.mass_action}
    R_j(x)= k_j \prod_{i=1}^n x_i^{\alpha_{ij}},
\end{equation}
(the so called Mass-Action kinetics), with the convention $0^0=1$, where $k_j, j=1,..,m$ are positive constants known as the reaction constants.\\
The stoichiometry coefficients are arranged in an $n \times \nu$ matrix $\Gamma=[\gamma_1 \ .. \gamma_n]^T$ called the \emph{stoichiometry matrix}, which is defined element-wise as:
\[    \Gamma _{ij} = \beta_{ij}- \alpha_{ij}.\]

Therefore, the dynamics of a CRN with $n$ species and $\nu$ reactions are described by a system of ordinary differential equations (ODEs) as:
\begin{equation}\label{e.ode}
    \dot x (t) = \Gamma R(x(t)), \ x(0) \in \Rnn
\end{equation}
where $x(t)$ is the concentration vector evolving in the nonnegative orthant  $\bar{\mathbb R}_+^n$, $\Gamma \in \mathbb R^{n \times \nu}$ is the stoichiometry matrix, $R(x(t))=[R_1(x(t)), R_2(x(t)), ..., R_\nu(x(t))]^T \in \bar{\mathbb R}_+^\nu$ is the reaction rates vector.

Note that \eqref{e.ode} belongs to the class of \emph{nonnegative systems}, i.e, $\Rnn$ is forward invariant. In addition, the manifold $\mathscr C_{x_\circ}:=(\{x(0)\}+ \mbox{Im}(\Gamma)) \cap \Rnn$ is forward invariant, and it is called \emph{the stoichiometric compatibility class} associated with $x_\circ$. Therefore, all stability results in this paper are relative to the stoichiometry compatibility class. \\
A left null vector $d \in \mathbb R^n, d^T \Gamma=0$ with $d>0$ is said to be a \emph{conservation law}, or a $P$-semiflow in petri-net literature terminology. If there exists a conservation law $d\gg 0$, the network is said to be \emph{conservative}.

Furthermore, the graph is assumed to satisfy:
\begin{enumerate}
 \item[\bf AG1.] There are no autocatalytic reactions, i.e., $\alpha_{ij}  \beta_{ij} =0$ for all $i=1,..,n, j=1,..,\nu$.
\item[\bf AG2.] There exists $v \in \ker \Gamma$ such that $v\gg0$. This condition is necessary for the existence of equilibria in the interior of stoichiometric compatibility classes.
\end{enumerate}
The set of reaction rate functions, i.e. kinetics, satisfying A1-A4 for a given $\Gamma$ satisfying AG1-AG2 is denoted by $\mathscr K_\Gamma$. A network family is the triple $(\mathscr S, \mathscr R,\mathscr K_\Gamma)$ which is denoted by $\mathscr N_\Gamma$.
\subsection{Graphical Representation}\label{sec.graphical}
A CRN can be represented via a bipartite weighted directed graph given by the quadruple $(\mathbf V_S, \mathbf V_R,\mathbf E,\mathbf W)$, where $\mathbf V_S$ is a set of nodes associated with species, and $\mathbf V_R$ is associated with reactions. The set of all nodes is denoted by $\mathbf V = \mathbf V_S \cup \mathbf V_R$. \\
The edge set $\mathbf E \subset \mathbf V \times \mathbf V$ is defined as follows. Whenever
a certain reaction $\R_j $ given by \eqref{e.reaction} belongs to $\mathscr R$
we draw an edge from $X_i \in \mathbf V_S$ to $\R_j \in \mathbf  V_R$ for all $X_i$'s such
that $\alpha_{ij} > 0$. That is, $(X_i ,\R_j) \in \mathbf E$ iff $\alpha_{ij} > 0$, and we say
in this case that $\R_j$ is an \emph{output reaction} for $X_i$. Similarly,
we draw an edge from $\R_j \in \mathbf  V_R$ to every $X_i \in \mathbf V_S$ such that
$\beta_{ij} > 0$. That is, $(\R_j, X_i) \in \mathbf E$ whenever $\beta_{ij} > 0$, and we
say in this case that $\R_j$ is an \emph{input reaction} for $X_i$. Notice that there can not be edges connecting two reactions or two species. Finally, $\mathbf W: \mathbf E \to \mathbb N$ is the weight function which associates to each edge a positive integer as $\mathbf W(X_i, \R_j) = \alpha _{ij}$, and $\mathbf W(\R_j,X_i)=\beta_{ij}.$ Hence, the stoichiometry matrix $\Gamma$ becomes the \emph{incidence matrix} of the graph. Figure \ref{f.crn_intro} depicts such a representation.
 \\A reaction $\R_{j_2}$ is called an \emph{ancestor} of $\R _{j_1}$ if there exists a directed sequence of edges $(\R_{j_2},X_{k_1})$, $(X_{k_1},\R_{k_1}),..,(X_{k_n},\R_{j_1})$ connecting them. The set of ancestors of $\R_j$ is denoted $\mathscr A(\R_j)$. Denote the set of indices of reactants of $\R_j$ by $M_j=\{ i | (X_i,\R_j ) \in E,1\le i\le n \}$, and let the set of indices of inflows be $\mathcal I=\{j|M_j=\varnothing,1\le j\le \nu\}$.\\
The set of output reactions of a set of species $P$ is denoted by $\Lambda(P)$. A nonempty set $P \subset  \mathbf V_S$ is called a \emph{siphon} \cite{Angeli07} if each input reaction
associated to $P$ is also an output reaction associated to $P$. A siphon is a \emph{deadlock} if $\Lambda(P)=\mathbf V_R$. A siphon or a deadlock is said to be \emph{critical} if it does not contain a set of species corresponding to the support of a conservation law. 

\section{PWLR Lyapunov Functions}

%
\subsection{Definition}
Consider a continuous Piecewise Linear (PWL) function $\tilde V$ defined over a polyhedral conic partition of $\mathbb R^\nu$. The partition is generated by a matrix $H \in \mathbb R^{p \times \nu}$, which is assumed to have some vector $\mu \in \ker H$ with $\mu\gg0$, and does not have zero rows. Let $\Sigma_1, ..., \Sigma_{2^p}$ be the set of $p \times p$ signature matrices, i.e. all possible $\{\pm 1\}$-diagonal matrices of size $p \times p$. Define cones $\cW_1, ..., \cW_{2^p}$ as:
\begin{align}\label{e.region}
 \cW_k = \{r \in \mathbb R^\nu: \Sigma_k Hr \ge 0 \}.
\end{align}

 $\cW_k$ can be seen as the intersection of half-spaces given by the inequalities $\sigma_{ki} h_i^T r\ge 0, i=1,..,p$, where $H=[h_1^T \ ... \ h_p^T]^T$, $\Sigma_k=\diag[\sigma_{k1} ... \sigma_{kp}]$. Note that these cones are not pointed as all of them contain the nontrivial $\ker H$. \\
As some of the intersections may have empty interiors, i.e., conflicting inequalities, we reorder the cones' indices such that the first $m$ cones are the nonempty-interior cones, i.e. $\cW_k^\circ \ne \varnothing$ iff $k \in \{1,..,m\}$. Note that $m\le 2^p$, where equality is achieved iff no left null vector of $H$ exists. Otherwise, the value of $m$ will depend on the number of left null vectors and their supports.

 Thus, we can state the following proposition which ensures the well-posedness of our subsequent definitions:
\begin{proposition}\label{prop1}
Let $H$, and $\{\cW_k\}_{k=1}^m$ be defined as above, then:
\begin{enumerate}
 \item \emph{Partitioning:} We have $\mathbb R^\nu =\bigcup_{k=1}^m \cW_k$, $\bigcap_{k=1}^m \cW_k=\ker H$,  and $\cW_k \cap \cW_j = \partial W_k \cap \partial \cW_j$, for all $j,k=1,..,m$.
\item \emph{Positivity:} All the cones intersect the positive orthant nontrivially, i.e., $\cW_k^\circ \cap \mathbb R_+^\nu \ne \varnothing, k=1,...,m$. Hence, $\mathbb R_+^\nu= \bigcup_{k=1}^m (\cW_k \cap \mathbb R_+^\nu)$ is a partition to non-empty-interior cones.
\item \emph{Symmetry:} For each $1\le k_1 \le m$, there exists $1 \le k_2 \le m$ such that $\Sigma_{k_1}=-\Sigma_{k_2}$. Hence, $m$ is always even and we can reorder the cones so that $\cW_{k}=-\cW_{m-k+1}, k=1,..,m/2$.
\end{enumerate}
\end{proposition}

After defining the partition, we are ready to define the function:
\begin{definition}\label{def.pwlr} Let $H$ with $\ker H=\ker \Gamma$, and $\{\cW_k\}_{k=1}^m$ be defined as above, and assume that $C=[c_1^T \ .. \ c_{m/2}^T]^T \in \mathbb R^{m/2 \times \nu}  $ be the \emph{coefficients matrix}. Then, $V: \mathbb R^n \to \mathbb R$ is said to be a Piecewise Linear in Rates (PWLR) function if it admits the representation $V(x)=\tilde V(R(x))$, where $\tV: \mathbb R^\nu \to \mathbb R$  is a continuous PWL function given as \vspace{-0.1in}
\begin{equation}\label{e.conLF}
 \tilde V(r) = |c_k^T r|, \ r \in \cW_k \cap -\cW_k, k=1,..,m/2.
\end{equation}
\end{definition}

Note that by definition, if $\tilde V(r)=c_k^T r$, then the function is defined over the region $\cW_k$, and if $\tilde V(r)=-c_k^T r$, then the corresponding region is $\cW_{-k}:=\cW_{m-k+1}=-\cW_{k}$.

Within the class of PWLR functions, the subclass of \emph{convex} PWLR functions admits a simpler representation:
\begin{definition}\label{def.cpwlr} Let $C=[c_1^T \ .. \ c_{m/2}^T]^T \in \mathbb R^{m/2 \times \nu}  $ be given such that there exists $v \in \ker C$ with $v\gg0$. Then, $V: \mathbb R^n \to \mathbb R$ is said to be a convex PWLR function if it admits the representation $V(x)=\tilde V(R(x))$, where $\tV: \mathbb R^\nu \to \mathbb R$  is a convex PWL given by
\begin{equation}\label{e.cvxLF}
 \tilde V(r)= \max_{1\le k \le m/2} |c_k^T r | = \| C r \|_{\infty}.
\end{equation}\end{definition}
\begin{remark} Although we name our proposed class of functions ``piecewise linear in rates functions'', the set of functions which are described by Definitions \ref{def.pwlr}, \ref{def.cpwlr} is a proper subset of the former. Therefore, note that we use this name because of convenience.
\end{remark}
\begin{remark}The matrices $H,C$ are crucial in the construction of the Lyapunov function. Methods for choosing them will be introduced in section IV.
\end{remark}

\begin{remark}It can be shown that any convex PWL function which satisfies \eqref{e.conLF} can be represented using \eqref{e.cvxLF} \cite{bereanu65}, and vice versa. Furthermore, given a function represented by \eqref{e.cvxLF}, the partition regions $\{\cW_k\}_{k=1}^m$ and the matrix $H$ can be determined.
\end{remark}

\begin{remark} The assumptions imposed on $\ker H$ are useful to simplify some Theorems and Algorithms. This will be later justified in Theorem \ref{cont.vs.cvx} where it will be shown that the existence of a Lyapunov function induced by a generic PWLR continuous function, i.e without assumptions on $H$, implies the existence of a corresponding convex PWL function. The $H$ associated with this convex function will satisfy the outlined conditions automatically.
\end{remark}

\subsection{Lyapunov Functions and Stability}
The main theme of this paper is to introduce a new class of Lyapunov functions for CRNs \eqref{e.ode} which are piecewise linear in terms of the reaction rates (PWLR). Therefore, the functions introduced in Definitions \ref{def.pwlr} and \ref{def.cpwlr} will be candidate Lyapunov functions. Such functions
  need to be non-increasing along the system's trajectories. However, since PWLR functions are non-differentiable on regions' boundaries, we will use the following expression of the time derivative of $V$ along the trajectories of \eqref{e.ode}:
\begin{align}\label{e.Vdot1}
 \dot V(x) := \max _{k \in K_x}  c_k^T \dot R(x),\end{align}
where $K_x = \{ k : c_k^T R(x) = V(x), 1\le k \le m \}$, $c_k=-c_{m+1-k}, k=1,..,\tfrac m2$,  and $\dot R(x)= \frac{\partial R(x)}{\partial x} \Gamma R(x) $. The justification for defining the time-derivative as above will be clear in the proof of Theorem \ref{th.lyap}.

 We define PWLR Lyapunov functions as follows:
\begin{definition} Given \eqref{e.ode} with initial condition $x_\circ:=x(0)\in \Rnn$. Let $V:\Rnn \to \bar{\mathbb R}_+$  be given as:  $V(x)=\tilde V(R(x))$, where $\tilde V$ is the associated PWL function. Then $V$ is said to be a \emph{PWLR Lyapunov Function} if it satisfies the following for all $R \in \mathscr K_\Gamma$,
\begin{enumerate}
 \item \emph{Positive-Definite}: $V(x)\ge0$ for all $x$, and $V(x)=0$ if and only if $R(x) \in \ker \Gamma$.
\item \emph{Nonincreasing}: $\dot V (x) \le 0$ for all $x$.
\end{enumerate}
The set of networks for which there exists a PWLR Lyapunov function is denoted by $\mathscr P$.
\end{definition}
\begin{definition}\label{def.lasalle}
A PWLR Lyapunov function for $\mathscr N_\Gamma$ is said to satisfy \emph{the LaSalle's condition} for $x_\circ$ if for all solutions $\tilde x(t)$ of \eqref{e.ode} with $ \tilde x(t) \in \ker \dot V \cap \mathscr C_{x_0}$ for all $t \ge 0$, we have $\tilde x(t) \in E_{x_\circ}$ for all $t \ge 0$, where $E_{x_\circ} \subset \mathscr C_{x_\circ}$ be the set of equilibria for \eqref{e.ode}.\end{definition}

The following theorem adapts Lyapunov's second method \cite{yoshizawa} to our context.
\begin{theorem}[Lyapunov's Second Method]\label{th.lyap}
  Given \eqref{e.ode} with initial condition $x_\circ \in \mathbb R_+^n$, and let $\mathscr C_{x_\circ}$ as the associated stoichiometric compatibility class. Assume there exists a PWLR Lyapunov function. and suppose that $x(t)$ is bounded,
  \begin{enumerate}
  \item then the equilibrium set $E_{x_\circ}$ is Lyapunov stable.
   \item If, in addition, $V$ satisfies the LaSalle's Condition, then
   $x(t) \to E_{x_\circ}$ as $t \to \infty$ (meaning that the point to set distance of $x(t)$ to $E_{x_\circ}$ tends to $0$). Furthermore, any isolated equilibrium relative to $\mathscr C_{x_\circ}$ is asymptotically stable.
   \end{enumerate}
 \end{theorem}

\begin{remark}
For a given $\Gamma$, the existence of a PWLR Lyapunov function establishes the stability of each system within the network family $\mathscr N_\Gamma$. Therefore, the Lyapunov function is robust with respect to all kinetic details of the network, and depends only on its graphical structure. It might seem that it is difficult for such function to exist, however, we will describe construction algorithms that are valid for wide classes of networks.
\end{remark}
\begin{remark} Note that the PWLR Lyapunov function considered can not be used to establish boundedness, as it may fail to be proper.  Therefore, we need to resort to other methods to guarantee boundedness a priori. For instance, if the network is conservative, i.e the exists $w \in \mathbb R_+^n$ such that $w^T\Gamma =0$, this ensures the compactness of $\mathscr C_{x_\circ}$.
\end{remark}

\begin{remark} The LaSalle's condition in Theorem \ref{th.lyap} can be verified  via a graphical algorithm to be described in \S\ref{sec.lasalle}.
\end{remark}

If the boundedness of solution was known a priori, then Theorem \ref{th.lyap} can be strengthened to the following:
\begin{corollary}[Global Stability]\label{cor1} Consider a CRN in $\mathscr P$ that satisfies the LaSalle condition with a given $x_\circ$. Assume that all the trajectories are bounded. If there exists $x^* \in E_{x_\circ}$, which is isolated relative to $\mathscr C_{x_\circ}$ then it is unique, i.e., $E_{x_\circ}=\{x^*\}$. Furthermore, it is globally asymptotically stable equilibrium relative to $\mathscr C_{x_\circ}$.
\end{corollary}
\begin{remark}
 Corollary \ref{cor1} implies that the existence of two or more isolated equilibria, even if the interior's equilibrium is unique, excludes the possibility of the existence of a PWLR Lyapunov function which satisfies the LaSalle's condition. This is to be contrasted with deficiency-zero theorem \cite{feinberg95} where boundary equilibria can be accommodated. This remark will be revisited in \S IV-A.
\end{remark}
\subsection{Checking candidate PWLR functions}
The first problem we shall tackle is that of checking whether a PWLR function is a Lyapunov function for a network family $\mathscr N_\Gamma$ given by $\Gamma \in \mathbb R^{n \times \nu}$.
In this subsection, we are given a candidate $V$ which is represented by the pair $C \in \mathbb R^{m/2 \times \nu}, H \in \mathbb R^{p \times \nu}$ as in \eqref{e.conLF}.

We need further notation. Fix $k\in\{1,..,m/2\}$. We claim that for proving the continuity of $V$, it is enough to test it between \textit{neighbors}, which we define next. Consider $H$, and for any pair of linearly dependent rows $h_{i_1}^T,h_{i_2}^T$ eliminate $h_{i_2}^T$. Denote the resulting matrix by $\tilde H \in \mathbb R^{\tilde p \times \nu}$, and let $\tilde\Sigma_1,..,\tilde\Sigma_m$ the corresponding signature matrices. Note that \eqref{e.region} can be written equivalently as $\cW_k=\{r|\tilde \Sigma_k \tilde Hr \ge 0\}$. The \textit{distance} $d_r$ between two regions $\cW_k,\cW_j$ is defined to be the \textit{Hamming distance} between $\tilde\Sigma_{k},\tilde\Sigma_j$. Hence, the set of neighbors of a region $\cW_k$, and the set of neighbor pairs are defined as:
\begin{align*}
 \mathcal N_k &= \{ j| d_r(\cW_j,\cW_k)=1, j=1,..,m\}, \\
\mathcal N &=\{ (j,k) | d_r(\cW_j,\cW_k)=1, 1\le j,k \le m \}
\end{align*}
Equivalently, a neighboring region to $\cW_k$ is one which differs only by the switching of one inequality. Denote the index of the switched inequality by the map $s_{k}(.):\mathcal N_k \to \{1,..,p\}$. For simplicity, we use the notation $s_{k\ell}:=s_k(\ell)$.  

 Let $c_k^T=[c_{k1} \ . . \ c_{k\nu}]$
 , and let
$ J_k = \supp(c_k)\subset \{1,..,\nu\}$ be the set of indices of reactions appearing in $c_k$. Define the set of indices of reactants of $J_k$ as follows
\begin{align}
 I_k=\{ 1 \le i \le n | \exists j \in J_k \,\mbox{such that} \, (X_i,\R_j) \in E. \}
\end{align}
Also, for all $i\in I_k$, define $J_{ki} = \{ j \in J_k | (X_i,\R_j) \in E \}$.

We are now ready to state the following theorem:
\begin{theorem}\label{th.checkcont}
 Let $\Gamma \in \mathbb R^{n \times \nu}$, and $C\in \mathbb R^{m/2 \times \nu}$ be given, and let $\tilde V$ be given by \eqref{e.conLF}. Then,
$ V(x) =   \tilde V(R(x))$ is a PWLR Lyapunov function for the network family $\mathscr N_\Gamma$ if and only if the following conditions hold:
\begin{enumerate}
 \item[\bf C1.] \emph{Nonnegativity:} For all $1\le k \le m/2$, there exists $\xi_k \in \mathbb R^p$ with $\xi_k>0$ such that $c_k^T=\xi_k^T \Sigma_k H$.
 \item[\bf C2.] \emph{Positive-Definiteness:} $\ker C = \ker \Gamma$.
      \item[\bf C3.] \emph{Continuity:} For all $(k,j)\in \mathcal N$, $\exists \eta_{kj} \in \mathbb R$ such that  \begin{equation}\label{e.continuety} c_k-c_j=\eta_{kj} h_{s_{kj}}. \end{equation} 
      \item[\bf C4.] \emph{Nonincreasingness:} Both the following holds:
     \begin{enumerate}
       \item[a)] For all $ k=1,..,m/2$, $i \in I_k$, we require $\sgn(c_{kj_1})\sgn(c_{kj_2})\ge 0$ for every $j_1,j_2 \in J_{ki}$. Thus, denote $\nu_{ki}=\sgn(c_{kj}), j \in J_{ki}$.
       \item[b)] There exists $\lambda^{(ki)} \in \mathbb R^{p}$, with $\lambda^{(ki)}\ge 0$ such that
  \begin{equation}\label{e.decreasing}
    -\nu_{ki} \gamma_i^T =  {\lambda^{(ki)}}^T \Sigma_k H ,
  \end{equation}
  where $c_{k}:=-c_{m+1-k}$ for $j=1+m/2,..,m$. Furthermore, if \eqref{e.decreasing} is satisfied, we shall choose $\lambda^{(ki)}$ so that $\supp(\lambda^{(ki)}) \subset s_k(\mathcal N_k)$.
     \end{enumerate}
       \end{enumerate}
Moreover, $\tV$ is convex if and only if $\eta_{kj}$'s can be chosen so that $\eta_{kj} \sigma_{ks_{kj}} \ge 0$.
\end{theorem}
\begin{remark}
 Note that C2 amounts to linear system solving, while C1,C3-C4 are equivalent to linear programming feasibility problems.
\end{remark}

\subsection{Checking candidate convex PWLR functions}
The conditions in the previous subsection will be simplified in the case of convex PWLR functions, as it can be noted that C1, C3 are satisfied automatically.   Consider $V$ with $\tilde V$ given by \eqref{e.cvxLF} with $\Gamma \in \mathbb R^{n \times \nu}$, and $C \in \mathbb R^{m/2 \times \nu}$ be given.


\begin{theorem} \label{th.checkcvx}
 Let $\Gamma \in \mathbb R^{n \times \nu}$, and $C\in \mathbb R^{m/2 \times \nu}$ be given. Then,
$ V(x) =   \| C R(x) \|_{\infty}$ is a PWLR Lyapunov function for the network family $\mathscr N_\Gamma$ if and only if the following two conditions hold:
\begin{enumerate}
 \item[\bf C2$'$.]\emph{Positive-Definiteness:} $\ker C = \ker \Gamma$.

 \item[\bf C4$'$.]\emph{Nonincreasingness:}  Both the following holds:
     \begin{enumerate}
       \item[a)]  For all $ k=1,..,m/2$, $i \in I_k$. We require $\sgn(c_{kj_1})\sgn(c_{kj_2})\ge 0$ for every $j_1,j_2 \in J_{ki}$. Thus, denote $\nu_{ki}=\sgn(c_{kj}), j \in J_{ki}$.
           \item [b)] There exist $\lambda^{(ki)} \in \mathbb R^m$, with $\lambda^{(ki)}\ge 0$ such that
  \begin{equation}\label{e.xdecreasing}
    -\nu_{ki} \gamma_i =  \sum_{\ell=1}^m \lambda_\ell^{(ki)} (c_k-c_\ell),
  \end{equation}
  where $c_{k}:=-c_{m+1-k}$ for $j=1+m/2,..,m$. Furthermore, if \eqref{e.decreasing} is satisfied, we shall choose $\lambda^{(ki)}$ with minimal support.
     \end{enumerate}
\end{enumerate}
\end{theorem}

\subsection{Existence of convex PWLR functions}
It can be noted that convex PWLR Lyapunov functions are easier to check and have stronger properties, therefore it is natural to ask whether the use of nonconvex counterparts is less conservative. In the case of linear systems, it is known that the existence of a nonconvex  Lyapunov function implies existence of a convex counterpart \cite{blanchini95}. Despite the nonlinear nature of our problem, next theorem shows that a similar result holds in our context:
\begin{theorem}\label{cont.vs.cvx} Let $\Gamma$ and \eqref{e.ode} be given, with the corresponding kinetics $\mathscr K_\Gamma$. If there exists a continuous PWLR Lyapunov function, then there exists a convex counterpart of the form \eqref{e.cvxLF}.
\end{theorem}
\begin{remark}The restrictions imposed on the kernel of $H$ in Definition \ref{def.cpwlr} are not needed to prove Theorem \ref{cont.vs.cvx}. \end{remark}

\subsection{Checking the LaSalle's Condition} \label{sec.lasalle}
In this section, we provide graphical algorithms for checking the LaSalle's condition stated in Theorem \ref{th.lyap}. Assume that a PWLR Lyapunov function exists. We use the same notation used in the previous two sections. Consider \eqref{e.decreasing} with $\lambda^{(ki)}$ chosen so that $\supp(\lambda^{(ki)}) \subset s_k(\mathcal N_k)$.  Let $L_{ki}=s_k^{-1}(\supp(\lambda^{(ki)}))$, which is nonempty since the LHS in \eqref{e.decreasing} is nonzero. Let $L_k=\bigcup_{i \in I_k} L_{ki} $
, $I_k^{(0)}:=I_k$, and $L_k^{(0)}:=L_k$. Define the following nested sets iteratively:
\begin{equation*}
 I_k^{(i)}= \bigcup_{\ell \in L_k^{(i-1)} } I_\ell, \  \mbox{and} \  L_k^{(i)}= \bigcup_{\ell \in L_k^{(i-1)} } L_\ell.
\end{equation*}
The iteration terminates when $L_k^{(i^*)}=L_k^{(i^*+1)}$. Denote $\bar I_k:= I_k^{(i^*)}$.
Using this notation, we state the following condition which we call the LaSalle's interior condition:
\begin{enumerate}
 \item[\bf C5i.] For all $k\in \{1,..,m\}$, the following shall hold: $\bar I_k=\{1,..,n\}$.
\end{enumerate}
In a nutshell the iterative process can be explained as follows: for every $k$, our aim is to show that $\dot x=0$ follows from the equality $c_k^T \dot R(x)=0$. Starting from the later equality, we get that the time derivative of species in $I_k$ vanish. Using \eqref{e.continuety},\eqref{e.decreasing}, this implies that $c_\ell ^T \dot R(x) =0$ for all $\ell \in L_k$. Using this procedure iteratively, we can expand the set of reactants whose derivative need to vanish. If the final set $\bar I_k$ is the whole set of species then this ensures that $\dot x=\Gamma R(x)=0$. \\
If the function $\tV$ is also convex, then the LaSalle's interior condition can be relaxed to:
\begin{enumerate}
 \item[\bf C5$'$i.] For all $k$, $c_k \in \im(\Gamma_{\bar I_k}^T)$, where $\Gamma_{\bar I_k}=[\gamma_{i_1}^T \, .. \, \gamma_{i_{o_k}}^T]^T, \bar I_k=\{i_1,..,i_{o_k}\}$ for some $o_k$.
\end{enumerate}

\begin{remark} As will be shown in the proof of Proposition \ref{th.lasalle}, conditions C5i and C5$'$i guarantee the LaSalle condition only provided $\tilde x(0) \in \mathbb R_+^n$, which explains the name. The LaSalle's interior condition alone can only establish the asymptotic stability of isolated equilibria in the relative interior of $\mathscr C_{x_0}$. In this case, Corollary \ref{cor1} will not hold, since solutions could in principle approach the boundary. However, if the persistence of the network can be verified a priori, for example by the absence of critical siphons \cite{Angeli07}, then the LaSalle's interior condition is enough to establish the result of Corollary \ref{cor1}.
\end{remark}

\begin{remark} C5$'$i does not follow from C2$'$ and C4$'$. For example consider the following CRN
\begin{align}\label{e.counter2}
     X_1 + X_2   \mathop{\lra}^{R_1} 0 \mathop{\lra}^{R_2} X_1, 0  \mathop{\lra}^{R_3} X_2,
\end{align}
 then $V(x)=|R_1(x)-R_2(x)|+|R_3(x)-R_2(x)|+|R_1(x)-R_2(x)|$ satisfies C2$'$ and C4$'$, but not C5$'$. It can be shown that there does not exist any $C$ satisfying the three conditions simultaneously, nor any pair $(H,C)$ satisfying C1-5.
\end{remark}

In order to strengthen the LaSalle's interior condition so that it applies to the boundary of stoichiometric compatibility class, we use the notion of \emph{critical siphons} defined in \S\ref{sec.graphical}. It has been shown in \cite{Angeli07
} that a face $\Psi$ of a stoichiometric compatibility class is invariant if and only if there exists a siphon $P$ such that $\Psi = \{ x \in \mathscr C_{x_\circ} | X_i \in P \Rightarrow x_i = 0   \}$. Since invariant faces arising from noncritical siphons correspond to  independent stoichiometric compatibility classes, we consider only critical siphons.  Let $\mathscr N_\Gamma$ be a given network, and $P_\ell$ be a critical siphon. We define the corresponding \emph{critical subnetwork} $\mathscr N_{\Gamma_\ell}$ as network with $\mathbf V_{R_\ell} = \mathbf V_R \backslash \Lambda(P_\ell) $, and $x_i(0)=0$ for $X_i \in P_\ell$. Furthermore, critical subnetworks of $\mathscr N_{\Gamma_{\ell}}$ are considered to be critical subnetworks of $\mathscr N_{\Gamma}$.  \\
We are now ready to state the LaSalle's condition:
\begin{proposition}\label{th.lasalle} Let $\mathscr N_\Gamma$ be a network, with a given PWLR Lyapunov function $V$. The network satisfies the LaSalle condition stated in Definition \ref{def.lasalle} for all non-negative $x_\circ$ if the following condition holds:
\begin{enumerate}
 \item[\bf C5.] The condition C5i (or C$'$5i if $\tV$ is convex) is satisfied for $\mathscr N_\Gamma$ and all its critical subnetworks.
\end{enumerate}
\end{proposition}
\section{Necessary Conditions for the Existence of PWLR functions}
As it is difficult to characterize exactly $\mathscr P$, i.e. the class of CRNs which admit a PWLR Lyapunov function, it is desirable to derive conditions which are necessary for a network to belong to $\mathscr P$. In this section, we will derive two conditions.

\subsection{Property of the Jacobian of $\mathscr P$ Networks}
We have shown in Corollary \ref{cor1} that networks in $\mathscr P$ with bounded trajectories and satisfying the LaSalle's condition can not have multiple isolated stoichiometrically compatible equilibria. In this subsection, we present a result along these lines by showing that the Jacobian matrix of networks belonging to $\mathscr P$ satisfies a property that has strong implications on uniqueness of equilibria and the injectivity of the map $F(x)=\Gamma R(x)$. \\
In order to introduce it, we need to define some notation. A matrix is said to be a $P$ matrix if all its principal minors are positive, and is said to be $P_0$ if all its principal minors are nonnegative. In particular, it encompasses $M$-matrices as a subclass. 
We state our next theorem as follows:

\begin{theorem}\label{th.p0}Given $\Gamma$. If $\mathscr N_\Gamma \subset \mathscr P$, then the Jacobian $-\Gamma \frac{\partial R}{\partial x}(x)$ is a $P_0$ matrix for all $x$, and for all networks in $\mathscr N_\Gamma$.
\end{theorem}

\begin{remark} It is known that a map is injective if its Jacobian matrix is $P$ \cite{nikaido65}. In our case, the Jacobian matrix being $P_0$ implies that the network can not admit multiple nondegenerate positive equilibria, where no assumption on boundedness is needed \cite[Appendix B]{banaji09}.
\end{remark}

\subsection{Constraints on the Possible Sign Patterns}
It is known that algorithms for checking that a given matrix is $P$ are exponential in time \cite{berman94}, therefore we provide in this subsection a weaker necessary condition which can be cast as a linear program. \\
 As mentioned in the proof of Theorem \ref{th.checkcont}, the nonpositivity of every term in the expansion of $\dot V$ is needed, and the sign of the derivative depends on the sign of $\dot x$. Hence, we partition $\mathbb R^\nu$ into \textit{sign regions} within which $\dot x$ has a constant term-wise sign. By AG2, we can define sign regions in an analogous way to \S 2.1, where we set $H=\Gamma$. Thus, we may write
\begin{align}\label{e.sign_regions}
 \mathcal S_k = \{r \in \mathbb R^\nu: \Sigma_k \Gamma r \ge 0 \}, \, k=1,..,m.
\end{align}
Note that the signature matrix $\Sigma_k$ specifies the sign of $\dot x$ in the region $\mathcal S_k$.
As a result, any linear-in-rates component $c_k^T R(x)$ operating on $\mathcal S_k$ must satisfy the term-wise sign constraint noted in \eqref{e.Vdot_exp}.
To encode this,  we need further notation.  Define the diagonal matrices $B_k = \diag [ b_{k1} \, ... \, b_{k\nu} ] $, $k=1,..,m/2$, where:
 \begin{equation}\label{e.sign_constraints} b_{kj} = \left \{ \begin{array}{rl} 1, & \mbox{if} \, M_j=\varnothing \\ 0, & \mbox{if}\ \exists {i_1,i_2\in M_j} \, \mbox{such that} \,  \sigma_{k i_1} \sigma_{k i_2} <0,  \\ -\sigma_{ji^*}, & \mbox{otherwise, for any} \, i^* \in M_j .    \end{array} \right . \end{equation}

Therefore, a linear program can be used to test the following necessary condition:
\begin{theorem}\label{th.necessary} Given $\Gamma$. Consider the network family $\mathscr N_\Gamma$, with $\{B_k\}_{k=1}^m$ defined as above, and let $U$ be a matrix whose columns form a basis for $\ker \Gamma$. If $\mathscr N_{\Gamma}$ admits a PWLR Lyapunov function, then there exists $0 \ne \zeta_k \in \mathbb R^\nu, k=1,..,m/2$ such that $ \zeta_k^T B_k U = 0 $, with $\zeta_{kj} \ge0 , j\in \{1,..,\nu\}\backslash \mathcal I.$
\end{theorem}

\begin{remark} As can be noted from the proofs, the necessary condition given by Theorem \ref{th.p0} is a consequence of the existence of a function $V$ that satisfies C1, C2, C4, i.e. it does not assume continuity. On the other hand, Theorem \ref{th.necessary} only assumes that $V$ exists satisfying conditions C2 and C4.
\end{remark}

Using Theorem \ref{th.necessary}, a simple graphical test for the nonexistence of a PWLR Lyapunov function can be derived. It can be stated as follows:
\begin{corollary}\label{cor.criticalSiphon} Given $\Gamma$. Consider the network family $\mathscr N_\Gamma$.
   If there exists a critical deadlock $P$, then $\mathscr N_{\Gamma} \not\in \mathscr P $.
\end{corollary}

\section{Construction of PWLR Lyapunov Functions}
 It is known that constructing convex PWL Lyapunov functions even in the case of systems evolving according to linear equations is not an easy task, and no simple necessary and sufficient conditions are available \cite{polanski95}.
In this section we propose several constructions of PWLR Lyapunov functions. Thus, we propose methods to find $(H,C)$ using the representation \eqref{e.conLF}, and $C$ using the convex representation \eqref{e.cvxLF}. The main difficulty, however, is that \eqref{e.continuety}, \eqref{e.decreasing}, \eqref{e.xdecreasing} are bilinear in the variables.




 \subsection{Construction of PWLR Lyapunov function over a given partition }
 \label{sec.LPalg}

Assume that the partition generator $\hat H$ is fixed, hence $\{\hat\cW_k\}_{k=1}^{m_h}$ is determined. Then, conditions C1,C3 are linear in $C$. Furthermore, the inclusion $\ker C \subset \ker \Gamma$ is implied by C1. The constraint \eqref{e.decreasing}, however, is nonconvex. Nevertheless, we shall rewrite it as a linear constraint.\\ Consider the sign regions $\{\mathcal S_k\}_{k=1}^{m_s}$ defined in \eqref{e.sign_regions}. If we intersect the two partitions  $\{\hat\cW_k\}_{k=1}^{m_h}$, $\{\mathcal S_k\}_{k=1}^{m_s}$ with the corresponding $c_k$'s inherited from $\{\hat\cW_k\}$. The matrix generating the new partition will be $H=[\Gamma^T \hat H^T]^T$. Therefore, we may consider, w.l.o.g, partitions induced by matrices of the form $H=[\Gamma^T \hat H^T]^T$, with corresponding sign matrices $\Sigma_k=\diag[\Sigma_k^{(s)} \ \Sigma_k^{(h)}]$. Note that we can consider $\{\cW_k\}_{k=1}^{m}$ as a refined partition of $\{\mathcal S_k\}_{k=1}^{m_s}$, hence we 
define the map $q(.): k \mapsto \ell $ if $\cW_k \subset \mathcal S_\ell$, and the notation $q(k)=q_k$ is used. Thus, we present the following:
\begin{theorem}\label{th.fixedH}
 Consider the system \eqref{e.ode}, with $H=[\Gamma^T \hat H^T]^T$, $\{\Sigma_k\}_{k=1}^m$,$\{B_k\}_{k=1}^{m_s}$, $q_k$ given as above. Consider the following linear program:
\begin{equation*}
\begin{aligned} 
&  {\text{Find}}
& & c_k, \xi_k, \zeta_k \in \mathbb R^\nu, \eta_{kj} \in \mathbb R, k=1,..,\tfrac m2; j \in \mathcal N_k,\\
& \text{subject to}
& & c_k^T=\xi_k^T \Sigma_k H ,\\ &&&
  c_k^T=\zeta_{k}^T B_{q_k}, \\ &&&
 c_k-c_j=\eta_{kj}\sigma_{ks_{kj}}h_{s_{kj}}, \\ &&&
\xi_k \ge0,   \mathbf 1^T \xi_k>0,  \\ &&&
\zeta_{kj} \ge0 , j\in \{1,..,\nu\}\backslash \mathcal I.
\end{aligned}
\end{equation*}
Then there exists a PWLR Lyapunov function with partitioning matrix $H$ if and only if there exist feasible solution to the above linear program with C2 satisfied. Furthermore, the PWLR function can be made convex by adding the constraints $\eta_{kj}\ge 0$.\end{theorem}
\begin{remark} A natural candidate for the partition matrix is $H=\Gamma$. Hence, we can write \[V(x)=c_k^T R(x)=\xi_k^{T} \Sigma_k \Gamma R(x) = \|\diag(\xi_k)\dot x \|_1, R(x) \in S_k.\] If we have additional constraint that for all $k$, $\xi_k=\mathbf 1$, then the Lyapunov function considered in \cite{maeda78},
\[ V(x)= \|   \dot x \|_1, \]
can be recovered as a special case. However, there are classes of networks for which $H=\Gamma$ does not induce a PWLR Lyapunov function, while there exists a partitioning matrix $\hat H$ which does. Understanding when this happens is a challenging open question.
\end{remark}

\subsection{Iterative Algorithm for Convex  PWLR functions}

In this subsection, we present an iterative algorithm for constructing convex PWLR Lyapunov functions. The idea is to start with an initial PWLR function, and aim for restricting the active region of each linear function $c_k^T R(x(t))$ to the region for which it is nonincreasing on it, i.e $ c_k^T \dot R(x(t))  \le 0 $. This is accomplished by adding extra linear components that satisfies certain properties.\\
Let $C_0 = [c_1 \ ... \ c_{m_0 } ] ^T  \in \mathbb R^{m_0  \times \nu}$, with the associated PWLR function. 
Define the \emph{active region} of a vector $ c_k$, $k=1,..,  {m_0} $, as:
\begin{equation*}
  \cW_0 (c_k) := \{r \in \mathbb R^\nu :  c_k^T r  \ge  c_j^T r  , - {m_0} \le k \le  {m_0},k \ne 0 \},
\end{equation*}
where $c_{-k}=-c_k$.
Assume that the associated CRN is given by \eqref{e.ode}. We define \emph{permissible region} of a linear component $ c_k $ to be the region for which it is nonincreasing. Hence,
\begin{align}\label{e.permissible} \nonumber
\mathcal P (c_k) &:=\{r \in \mathbb R^\nu :  \nu_{ki} \gamma_i ^T r  \le 0, i \in I_k \}  \\ &\subset    \{ \tilde r \in \mathbb R^\nu :\tilde r=R(x),  c_k^T \textstyle\frac{\partial R}{\partial x}(x) \Gamma R(x)  \le 0 \}, \vspace{0em}
\end{align}
where $\nu_{ki}=\sgn(c_{ki})$.
Note that in general, $\cW_0 (c_k) \not \subset \mathcal P(c_k)$. Therefore, we need to define a new PWL function with matrix $C_1$ so that $\cW_1 (c_k) \subset \mathcal P(c_k) $. To achieve this, we augment new rows to $C_0$. The new rows are of the form \vspace{0em}
\begin{equation}\label{e.alg_new} c_{m_0+i} := c_k + \nu_{ki} \gamma_i, i\in I_k. \vspace{0em} \end{equation}
Thus, $C_1^{'} := [C_0^T \ c_{m_0+1} \ ... \ c_{m_0+n} ] ^ T $. Finally, $C_1$ is defined by  eliminating linearly dependent pairs of rows from $C_1^{'}$.

 Hence, Algorithm 1 can be described as:
 \begin{enumerate}
   \item Given $C_0 = [c_1 \ ... \ c_{m_0} ] ^T  \in \mathbb R^{m_0/2 \times \nu}$, $k=1,..,m_0/2$, and $\ker \Gamma \subset \ker C_0$. Set $k=1$.
   \item Define $C_k^{'} := [C_{k-1}^T \ c_{m_k+1} \ ... \ c_{m_k+n} ] ^ T $, where $c_{m_k+i} := c_k + \nu_{ki} \gamma_i, i=1,...,n.$
   \item Define $C_k$ as $C_k^{'}$ with linearly dependent pairs of rows eliminated.
   \item If $C_{k}=C_{k-1}$ or $k>N$, stop.
   \item Set $k:=k+1$, and go to step 2,
 \end{enumerate}
 where $N$ is the maximum number of iterations allowed. \\
 If Algorithm 1 terminates then we state the following:

\begin{theorem}\label{th_alg}Consider \eqref{e.ode}. If Algorithm 1 terminates after finite number of iterations with C2$'$ with satisfied, then the resulting function is a PWLR Lyapunov function for the network family $\mathscr N_\Gamma$.
  \end{theorem}

\begin{remark} The formula \eqref{e.alg_new} is not the unique way for constructing new vectors. Indeed, one can replace the inequality $\nu_{ki} \gamma_i\le 0$ with any system of inequalities covering the same region. For instance, the region defined by the inequality $R_1- 2R_2 + R_3 \le 0$ is a subset of the region defined by the pair $ R_1-R_2 \le 0 , -R_2+R_3 \le 0$. Therefore, the \emph{standard setting} of Algorithm 1 means using \eqref{e.alg_new} with $C=\Gamma$. \vspace{-0.7em}
\end{remark}
\subsection{Special Constructions}
It is possible to construct PWLR Lyapunov functions for CRNs with specific structure. We state the following result which enjoys having easy-to-check graphical condition:
\begin{theorem}\label{th.maxmin}
 Consider the network family $\mathscr N_\Gamma$. Suppose the following properties are satisfied:
\begin{enumerate}
 \item $\dim(\ker \Gamma)=1$,
 \item $\forall X_i \in \mathbf V_S$, there exists a unique output reaction, i.e every row in $\Gamma$ has a unique negative element,
\end{enumerate}
Then,
\begin{enumerate}
\item the following is a PWLR function for the network family $\mathscr N_\Gamma$:
\begin{align}\label{e.maxmin}
 V(x)= \max_{1\le j \le \nu } \frac 1{v_j} R_j(x) -  \min_{1\le j \le \nu } \frac 1{v_j} R_j(x),
\end{align}
where $v =[ v_1 \, ... \, v_\nu ]^T \in \ker (\Gamma), v \gg 0 $.
\item LaSalle's interior condition holds if  $\mathscr A(\R_j) \cap \mathscr A(\R_\ell) \ne \varnothing$, for all $1\le j,\ell \le \nu $.
\item  If the network is conservative, then it is persistent, i.e, $\omega(x_0) \cap \partial \mathbb R_+^n = \emptyset$ for all $x_\circ$.
Furthermore, if there exists an isolated equilibrium, then it is a unique globally asymptotically stable equilibrium with respect to $\mathscr C_{x_\circ}$.
\end{enumerate}
\end{theorem}
Theorem \ref{th.maxmin} can be extended to allow the addition of the reverse of certain reactions. Note that adding the reverse of an irreversible reaction increases the dimension of the kernel of $\Gamma$ so that the original result would not normally apply.

\begin{theorem}\label{th2} Consider the network $\mathscr N_\Gamma$ with the associated graph $(\mathbf V_S,\mathbf  V_R,\mathbf E,\mathbf W)$ that satisfies the conditions of Theorem \ref{th.maxmin}. Let $ \mathbf V_{R'} \subset \mathbf  V_R$
be the set of reactions $\R_j$ that satisfy: if $(\R_j,S_i) \in E $ and $(\R_k,S_i) \in E$ then $j=k$. Equivalently, $\R_j \in {\bf V}_{R'}$ if it is the only input reaction for all of its product species, i.e. the corresponding column in $\Gamma$ does not have more than one positive element.  Let $(\mathscr S, \tilde {\mathscr R})$ be the CRN constructed by adding reverse reactions for reactions belonging to $\mathbf V_{R'}$, and let $\tilde \Gamma$ be the new stoichiometry matrix. \\
Then, the claims of the previous theorem are satisfied for $\mathscr N_{\tilde \Gamma}$ with the following function:
\begin{align}\label{e.lyap_r}
    &V(x) = \max_{1\le j\le \nu} \tfrac 1{v_j}( R_j(x) -\chi_j R_{-j}(x)) -\\ \nonumber & \qquad \qquad\qquad\qquad \min_{1\le j\le \nu} \tfrac 1{v_j}( R_j(x) -\chi_j R_{-j}(x)),
\end{align}
where $\chi_j=1$ if ${\rm \bf R_j} \in \tilde {\mathscr R}\backslash {\mathscr R}$, and  $\chi_j=0$ otherwise.
\end{theorem}

\section{Discussion and Examples}

\subsection{Relationship to Consensus Dynamics}
Consider a closed, i.e. without inflows or outflows, CRN for which there is a unique reactant for every reaction, and a unique output reaction for every species. In such network, the bipartite graph representing this network can be replaced with a digraph $G=(\mathbf V,\mathbf E,\mathbf W)$ representing the species, reactions and weights respectively. The stoichiometry matrix $\Gamma$ will be the negative transpose of the \emph{Laplacian} of the digraph. Hence, CRN can be described by the ODE:
\begin{equation}\label{e.linear}
  \dot x = - L^T R(x).
\end{equation}
If the graph is strongly connected, then $\mathbf 1$ is a conservation law, i.e, $\mathbf 1^TL^T=0$. Using Perron-Frobenius theory \cite{berman94}, $\ker L^T$ is spanned by a unique vector $v\gg 0$. Hence, \eqref{e.maxmin} is a Lyapunov function for the network family $\mathscr N_{-L^T}$ by Theorem \ref{th.maxmin}. Note that this is very similar to a consensus algorithm in a network of agents \cite{murray07} where they consider algorithm of the form: $\dot x = -L x$.  Indeed, we can derive from Theorem \ref{th2} the following result for consensus algorithms:

\begin{corollary}\label{cor.consensus} Consider a network of $n$ integrator agents with a strongly connected digraph $G=(\mathbf V,\mathbf E,\mathbf W)$, and let $L=[\ell_{ij}]$ be the associated Laplacian. Consider applying the following consensus algorithm:\vspace{0em}
\begin{equation}\label{e.consensus}
  \dot x = -L F(x),\vspace{0em}
\end{equation}
where $F=[F_1,..,F_n]^T$ is any function that satisfies: there exists $F^+,F^- \in \mathscr K_{-L}$ with $F=F^+-F^-$, and $F_j^- \equiv 0$ if there exists more than one positive element in the $j^{\text{th}}$-column of $-L$. Then, $F$-consensus is asymptotically reached for all initial states, i.e, $\lim_{t \to \infty} F(x_1(t)) = ..=\lim_{t \to \infty} F(x_n(t)) < \infty$.
\end{corollary}
\begin{remark}
Note that a mix-min type Lyapunov function \eqref{e.maxmin} has been already used for linear consensus algorithms \cite{moreau04}. Therefore, Corollary \ref{cor.consensus} generalize the results of \cite{moreau04,murray07} to nonlinear consensus algorithms. It is worth noting that the dynamics of a detailed balanced network has also been linked with consensus dynamics \cite[\S 4.4]{schaft_arxiv}.\vspace{-1em}
\end{remark}
\subsection{Illustrative Examples}
We present several examples to illustrate the results:
\begin{enumerate}
\item Consider the network (1),(2) introduced in the introduction. As indicated before, this network does not satisfy the conditions of \cite{feinberg95} as it has deficiency 1, and violates the conditions of \cite{maeda78,angeli10}. Hence, its stability can not be established by methods in the literature.\\
           The network has two conservation laws. Thus, the stoichiometric class is a two dimensional polytope of the form $\mathscr C_{x_\circ}=\{x \in \mathbb R^4 | x_1+x_3=M_1, x_1+x_2+2x_4=M_2, 0 \le x_1 \le M_1, 0 \le x_1+x_2 \le M_2 \}$, where $M_1=x_{\circ 1} +x_{\circ 3} , M_2 =x_{\circ 1} +x_{\circ 2}+2x_{\circ 4} $ are the conserved quantities.

    In order to apply Theorem \ref{th.fixedH}, let us choose $H=\Gamma$. Then, there are six non-empty-interior partition regions of the reaction space $\mathbb R^3$, which are: \begin{align*}\cW_1&=\{ r | -h_1^T r \le 0, -h_2^T r \le 0, h_3^T r \le 0, h_4^T r \le 0\}, \\
      \cW_2&=\{ r | -h_1^T r \le 0, h_2^T r \le 0, h_3^T r \le 0, -h_4^T r \le 0\}, \\
       \cW_3&=\{ r | -h_1^T r \le 0, h_2^T r \le 0, h_3^T r \le 0, h_4^T r \le 0\},\\ \cW_4 &=-\cW_3, \cW_5=-\cW_2, \cW_6=-\cW_1, \end{align*}
       where $h_1=[-1 , 0 , 1 ]^T, h_2=[1 , -\! 2  , 1 ]^T, h_3=[1, 0 , -1 ]^T, h_4=[0 , 1 ,-\! 1 ]^T$. \\ We need to find the coefficients $c_1,..,c_6$, where $c_4=-c_3, c_5=-c_2, c_6=-c_1$.
        Although we have twelve neighboring pairs, only three constraints are needed, because of the symmetries involved, which are $c_3-c_1=\eta_{31} h_2, c_3-c_2=\eta_{32} h_4, c_2+c_1=-\eta_{21}h_1$. The sign-constraints vectors are $b_1=[ 1,1,-\!1]^T,b_2=[ 1,-\!1 , 0]^T, b_3=[1,-\!1,-\! 1]^T,b_4=-  b_3, b_5=-b_2, b_6=-b_1$. Hence, the linear program can be solved and one of its solutions is $V(x)=\tV(R(x))$, where $\tV$ is:
        \begin{equation*}\label{e.LP_lyap}
          \tV(r)\!=\!\max\{ |r_1+3 r_2 -4r_3|, 3|r_1-r_2|, |3r_1 - r_2 - 2r_3 | \}.
        \end{equation*}

        Alternatively, applying Algorithm 1 with the standard setting yields a PWLR Lyapunov function given by:
        \begin{align*} &\tV(r)= \\ &\max\{|r_1-r_3|, \! |r_1-2r_2+r_3|, \! 2|r_2-r_3|, \! 2|r_2-r_1| \}. \end{align*}

        Finally, Theorem \ref{th.maxmin} gives \eqref{e.lyap_eg}. Therefore, our three constructions were successful and have produced three different functions. It can be verified that the LaSalle's condition is fulfilled. Since the network is conservative and injective relative the stoichiometric class \cite{banaji07}  there exists a unique equilibrium in each stoichiometric compatibility class. Therefore, Corollary \ref{cor1} implies that the unique equilibrium is globally asymptotically stable.  In order to illustrate the dynamics, we consider the stoichiometric class corresponding to $M_1=8, M_2=7$. Figure \ref{fig.phase_portrait} depicts the level sets of Lyapunov function \eqref{e.lyap_eg} and the phase portrait with 
        2$^{\text{nd}}$-order Hill kinetics which are given by: $R(x)=[k_1 x_1^2/(1+x_1^2), k_2 x_2^4/(1+x_2^2)^2, k_3 x_3^2 x_4^2/( (1+x_3^2)(1+x_4^2)) ]^T$ where the rate constants are $k=[1,0.5,0.25]$.

\begin{figure}
  \centering
    \includegraphics[width=\columnwidth]{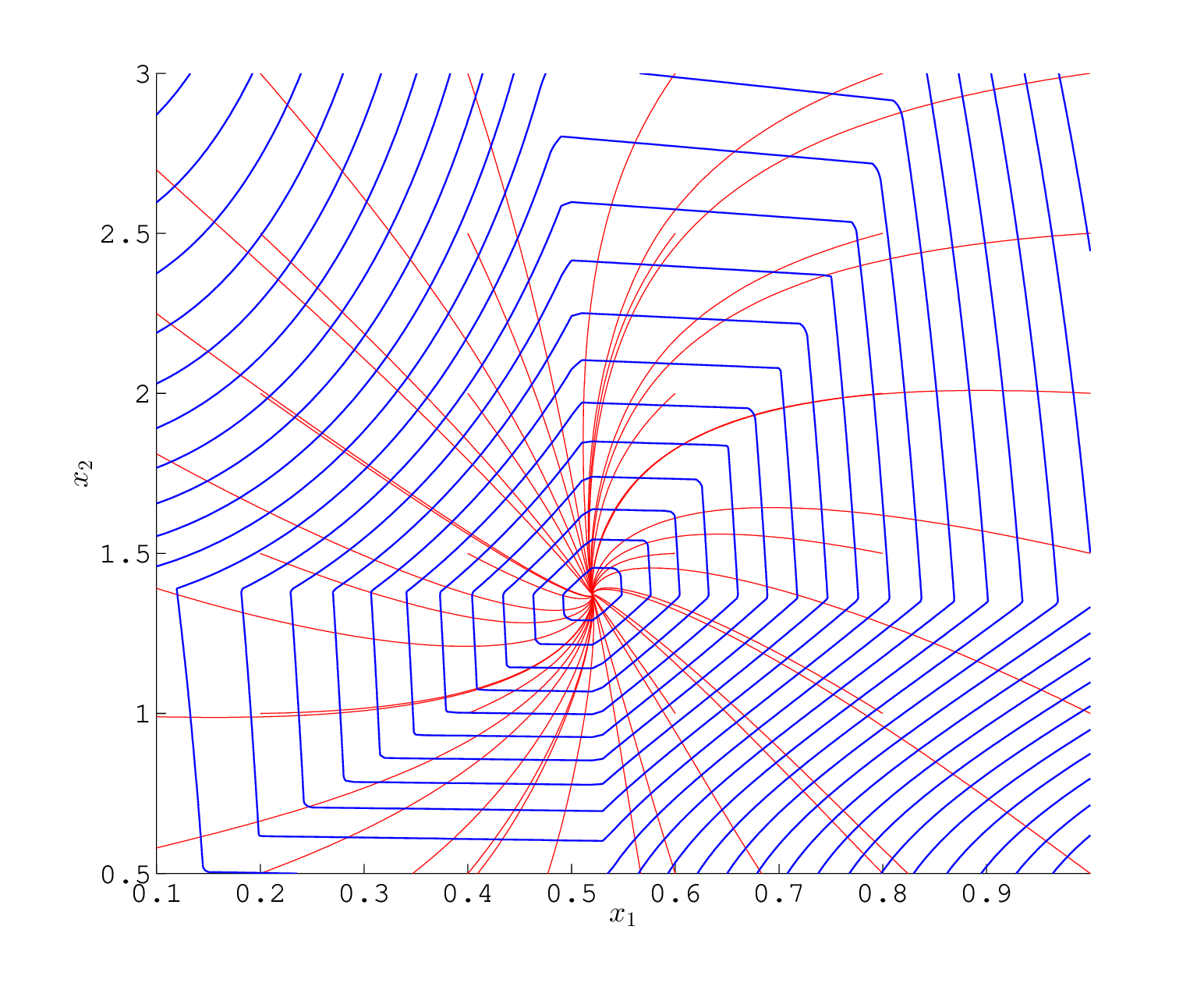}
  \caption{Lyapunov function level sets and phase portrait with
   2$^{\text{nd}}$-order Hill kinetics.}\label{fig.phase_portrait}
\end{figure}

 \item Consider the following network given in \cite{feinberg95}:\vspace{0em}
\begin{equation*}
X_1 \leftrightharpoons 2X_2, \
X_1+X_3   \leftrightharpoons X_3 \lra X_2+X_5 \lra X_1+X_3  \vspace{0em}     \end{equation*}
The network violates both necessary conditions given by Theorems \ref{th.p0}, \ref{th.necessary}, therefore it does not admit a PWLR Lyapunov function. However, the deficiency-zero theorem \cite{feinberg95} can be applied with Mass-Action kinetics to show that the interior equilibrium is asymptotically stable despite the existence of boundary equilibria, a situation which is not allowed by Theorem \ref{th.p0}.

\item We gave an example of a zero deficiency network which is not $\mathscr P$. Now, consider the following CRN for a given integer $n\ge 1$:
    \begin{equation*}\!\!\!
    \begin{array}{ll}
      X_1+E_1 \lra E_1X_1 \lra  X_2+E_1,& \hspace{-0.8cm} X_2 \lra X_1 \\
      X_2+E_2 \lra E_2X_2 \lra  X_3+E_2,& \hspace{-0.8cm}  X_3 \lra X_2 \\
      \vdots \\
       X_n+E_n \lra E_nX_n \lra  X_{n+1}+E_n,& \!\!\!\!\!  X_{n+1} \lra X_n,
       \end{array}
    \end{equation*}
which has deficiency $n$. For every $n$, a PWLR Lyapunov function is given by $V(x)=\|D \dot x\|_1$ where $D=\diag[I_{2n+1},O_{n}]$, with species ordered as $X_1,..,E_1X_1,..,E_1,..,E_n$. This shows that there is no simple relationship between our results and the notion of deficiency.

\item The following CRN illustrates the fact that the mere existence of the PWLR Lyapunov function does not guarantee the boundedness of the trajectories: \vspace{0em}
\begin{equation*}
     X_3   \mathop{\lra}^{k_1} X_1, \  0 \mathop{\lra}^{k_2} X_2, X_1+X_2 \mathop{\lra}^{k_3} X_3,\vspace{0em}
        \end{equation*}
        The three constructions presented yield a Lyapunov function, in particular \eqref{e.maxmin} is a valid one. However, consider the network with Mass-Action Kinetics, and let $A=x_1(0)+x_3(0)$ be the parameter corresponding to the stoichiometric compatibility class. If $A>\tfrac{k_2}{k_3}$, then the system trajectories are bounded and the unique equilibrium $\left ( \frac{k_2k_3}{ k_3 A - k_2}, A- \frac {k_2}{k_3}, \frac {k_2}{k_3} \right)$ is globally asymptotically stable by Theorem \ref{th.lyap}. However, when $A \le \tfrac{k_2}{k_3}$, there are no equilibria in the nonnegative orthant, solutions are unbounded and approach the boundary.

        \item Consider the following network: \vspace{0em}
          \begin{equation*}
     X_1   \mathop{\lra}^{k_1} X_2,  \ X_5 \mathop{\lra}^{k_4} X_4, \
     X_2 + X_4   \mathop{\lra}^{k_2} X_3 \mathop{\lra}^{k_3} X_1+X_5 \vspace{0em}
        \end{equation*}
        The linear program in Theorem \ref{th.fixedH} with $H=\Gamma$ is infeasible, however, Theorem \ref{th.checkcont} and Theorem \ref{th.maxmin} give rise to the PWLR function \eqref{e.maxmin} with $v=\mathbf 1$. Close examination indicates a partitioning matrix $\hat H=[1 \, 0 \, 0 \,  -\!\!1 ] $ renders the linear program feasible.

        \item Consider the following network:\vspace{0em}
         \begin{equation*}
     2X_1+3X_3  \mathop{\lra}^{k_1} 0    \mathop{\lra}^{k_3} 3X_1+X_2+2X_3, \,
     X_1 + X_2  \mathop{\lra}^{k_2} X_3 \vspace{0em}
        \end{equation*}
        Theorem \ref{th.maxmin} does not apply. Algorithm 1 with standard setting does not terminate. However, Theorem \ref{th.fixedH} with $H=\Gamma$ gives the following convex PWLR Lyapunov function:
       $ V(x)= \max \{ |6R_1(x)+R_2(x) - 7 R_3(x)|,   |3R_2(x)-3R_3(x)|, |6R_1(x)- 6R_3(x) | \}.$ \vspace{0em}
\end{enumerate}
\subsection{Biochemical Example}
Within the class of structurally persistent, i.e. critical-siphon-free, CRNs which have a $P_0$ Jacobian matrix, our proposed algorithms were reasonably successful. As an example, consider the following CRN which represents a double futile cycle with distinct enzymes \cite{angeli08}:
\begin{align*}
    X_0 +E_0   \mathop{\leftrightharpoons}^{k_1}_{k_{-1}} E_0X_0 \overset{k_2}{\lra} X_1+E_0, \\
    X_1+E_1  \mathop{\leftrightharpoons}^{k_3}_{k_{-3}} E_1X_1\overset{k_4}{\lra} X_0 + E_1,  \\ \nonumber
      X_1 +F_0  \mathop{\leftrightharpoons}^{k_5}_{k_{-5}} F_0X_1 \overset{k_6}{\lra} X_2+F_0, \\
    X_2+F_1   \mathop{\leftrightharpoons}^{k_7}_{k_{-7}} F_1X_2\overset{k_8}{\lra} X_1 + F_1,
        \end{align*}
        where the associated graph is depicted in Figure \ref{f.d_futile}.

The network is conservative with five conservation laws, hence the stoichiometric space is a 6-dimensional compact polyhedron.

Both Theorems \ref{th.fixedH}, \ref{th_alg} are applicable.
        For example, a valid PWLR Lyapunov function constructed can be represented as:
       $V(x) = \| \diag(\xi) \dot x \|_1, $
       where $\xi=[2 \, 2 \, 2 \, 1 \, 1 \, 1 \, 1 \, 1 \, 1 \, 1 \, 1 ]$ and species are ordered as $X_0, X_1, X_2, \ldots, F_1 X_2$.
The network is injective by the work of \cite{banaji07}, hence it can not have more than a single equilibrium state in the interior of each stoichiometric class. Furthermore, it has deficiency 2, hence the zero-deficiency theorem will not apply. Also, the results of \cite{angeli10} can not be applied since $X_1$ is adjacent to more than two reactions. However, Theorem \ref{th.lyap} implies that a Lyapunov function exists and that the unique equilibrium is globally asymptotically stable. Figure \ref{f.biochemical} depicts a sample trajectory with Michaelis-Menten kinetics of the form: $ R_j(x)= k_j \prod_i (x_i/(a_{ij}+x_i))^{\alpha_{ij}}, $  with $a_{ij}=1$, and kinetic constants $k$=[33.2, 83.97, 37.17, 82.82, 17.65, 12.95, 87.99, 4.41, 68.67, 73.38, 43.72, 37.98] and initial condition $x_\circ$=[5.88, 8.78,  4.69, 4.37, 7.46, 4.68, 8.61, 4.67, 4.98, 4.87, 2.29].

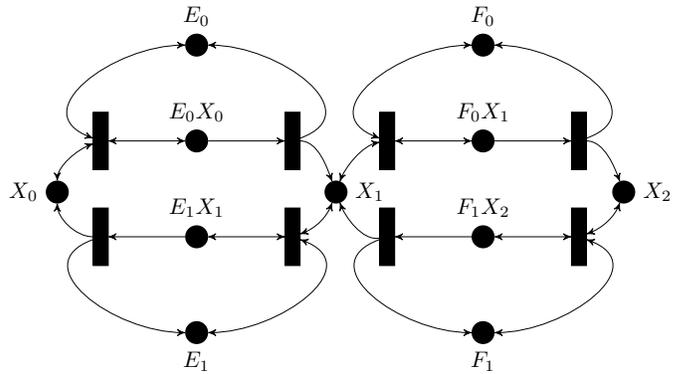
\begin{figure}
\centering
\scalebox{0.85}{
\centering
\begin{tikzpicture}[node distance=1.5cm,>=stealth',bend angle=45,auto]
  \tikzstyle{S}=[circle,thick,draw=black,fill=black,minimum size=2.5mm]
  \tikzstyle{R}=[rectangle,very thin,draw=black,
  			  fill=black,minimum width=.01in, minimum height=9mm]
\begin{scope}
   \node [S] (X1) [label=above:$E_0X_0$]   {};
    \node [R] (R1) [right of=X1] {}
      edge [<-]                  (X1);
    \node [R] (R2) [left of=X1] {}
      edge [<->]                  (X1);
        \node [S] (X3) [below of=X1, label=above:$E_1X_1$] {};
    \node [R] (R3) [right of=X3] {}
      edge [<->]                  (X3);
    \node [R] (R4) [left of=X3] {}
      edge [<-]                  (X3);
       \node [S] (X2) [below=0.8cm, right=2cm, label=right:$X_1$]   {};
        \node [S] (X4) [below=0.8cm, left=2cm, label=left:$X_0$]   {};

         \node [S] (X5) [above of=X1, label=above:$E_0$]   {};
    \node [S] (X6) [below of=X3, label=below:$E_1$]   {};

    \node[S](X7)[right=4.3cm,label=above:$F_0X_1$]{};
     \node [R] (R5) [right of=X7] {}
      edge [<-]                  (X7);
    \node [R] (R6) [left of=X7] {}
      edge [<->]                  (X7);

            \node [S] (X9) [below of=X7, label=above:$F_1X_2$] {};
    \node [R] (R7) [right of=X9] {}
      edge [<->]                  (X9);
    \node [R] (R8) [left of=X9] {}
      edge [<-]                  (X9);
            \node [S] (X10) [above of=X7, label=above:$F_0$]   {};
    \node [S] (X11) [below of=X9, label=below:$F_1$]   {};
      \node [S] (X8) [below=0.8cm, right=6.5cm,, label=right:$X_2$]   {};

     \draw[->] (R1) .. controls +(0:0.4cm) and +(110:0.4cm) .. (X2);
     \draw[<->] (X2) .. controls +(-110:0.4cm) and +(20:0.4cm) .. (R3);
      \draw[->] (R4) .. controls +(-180:0.4cm) and +(-90:0.4cm) .. (X4);
   \draw[<->] (X4) .. controls +(+90:0.4cm) and +(-160:0.4cm) .. (R2);
      \draw[->] (R1) .. controls +(20:1.3cm) and +(0:1.3cm) .. (X5);
      \draw[<->] (X5) .. controls +(-180:1.3cm) and +(160:1.3cm) .. (R2);
         \draw[->] (R4) .. controls +(200:1.3cm) and +(180:1.3cm) .. (X6);
      \draw[<->] (X6) .. controls +(0:1.3cm) and +(-20:1.3cm) .. (R3);

          \draw[->] (R8) .. controls +(-170:0.4cm) and +(-70:0.4cm) .. (X2);
          \draw[<->] (X2) .. controls +(70:0.4cm) and +(-170:0.4cm) .. (R6);
          \draw[->] (R5) .. controls +(20:1.3cm) and +(0:1.3cm) .. (X10);
      \draw[<->] (X10) .. controls +(-180:1.3cm) and +(160:1.3cm) .. (R6);
              \draw[->] (R8) .. controls +(200:1.3cm) and +(180:1.3cm) .. (X11);
      \draw[<->] (X11) .. controls +(0:1.3cm) and +(-20:1.3cm) .. (R7);
          \draw[->] (R5) .. controls +(0:0.4cm) and +(110:0.4cm) .. (X8);
     \draw[<->] (X8) .. controls +(-110:0.4cm) and +(20:0.4cm) .. (R7);
   \end{scope}
\end{tikzpicture}}
\caption{Double Futile Cycle with distinct enzymes.}\label{f.d_futile}
\end{figure}

\begin{figure}
  \centering
  \includegraphics[width=\columnwidth]{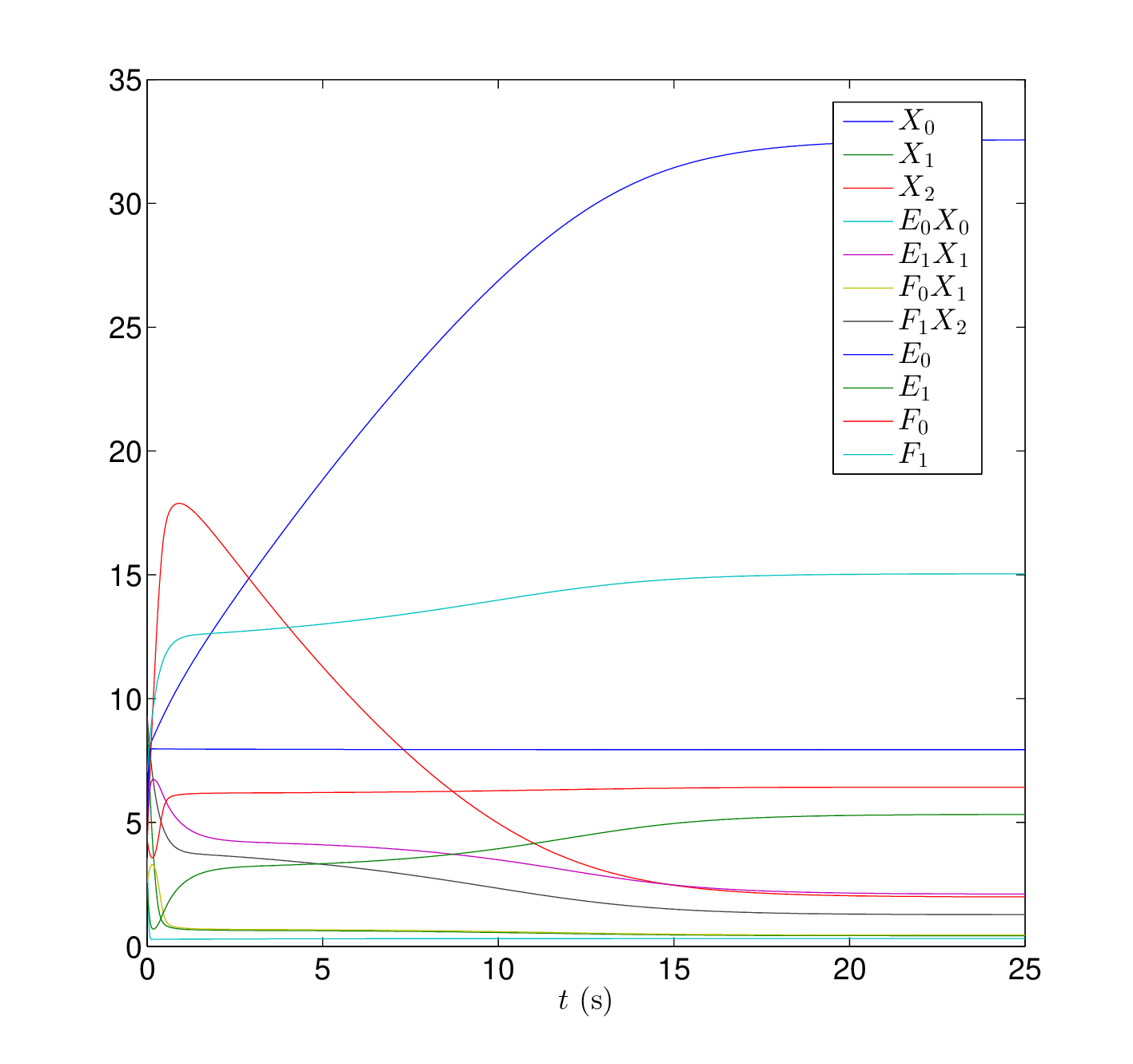}\\
  \caption{Sample trajectories for the biochemical example with Michaelis-Menten kinetics.}\label{f.biochemical}
\end{figure}

\section{Conclusions}
 A new type of Lyapunov functions have been introduced for a wide class of CRNs. The Lyapunov functions are piecewise linear and possibly convex in terms of monotone reaction rates. We have provided methods for checking candidate PWLR Lyapunov functions. Several theorems were introduced for their construction.

Concerning potential future directions, we mention few. First, further properties of the class of $\mathscr P$-networks, specifically the persistence property, are currently being investigated. Second, a more general class of robust Lyapunov functions, of which PWLR functions are a special case, are being studied. Third, the potential use of PWLR Lyapunov functions as control Lyapunov functions is being researched. 

\section*{Appendix: Proofs}
  \addcontentsline{toc}{section}{Appendix: Proofs}
\paragraph*{Proof of Proposition \ref{prop1}}
a) By construction, we have  $\mathbb R^\nu = \bigcup_{k=1}^{2^p} \cW_k $. To prove the claim it is sufficient to show that if $\cW_k^\circ = \varnothing$, then $\exists j \in \{1,..,m\}$ such that $\cW_k \subset \cW_j$. By definition, $\cW_k= \bigcap_{i=1}^p \{r | \sigma_{ki} h_i^T r \ge 0\}$. We construct the set $\cW_j$ as follows: Let $\cW_j^{(1)}=\{r| \sigma_{k1} h_1^T r \ge 0\}$ which has nonempty interior. Let $i>1$, we set $\cW_j^{(i)} = \cW_j^{(i-1)} \cap \{r| \sigma_{ki} h_i^T r \ge 0\}$ if it has nonempty interior.
Otherwise, $\cW_j^{(i-1)} \cap \{r| \sigma_{ki} h_i^T r \geq 0 \} = \cW_j^{(i-1)} \cap \{r| \sigma_{ki} h_i^T r = 0 \} \subseteq  \cW_j^{(i-1)} \cap \{r| \sigma_{ki} h_i^T r\leq 0 \}:= \cW_j^{(i)}$ and the latter will have
 nonempty interior. Therefore, $\cW_j=\cW_j^{(p)}$ will have nonempty interior and will include $\cW_k$. Furthermore, it is evident from the definitions that $\bigcap_{k=1}^m \cW_k=\ker H$, $\cW_j \cap \cW_k = \partial \cW_k \cap \partial\cW_j$, $j \neq k \in \{1,..,2^p\}$. \\
b) Let $ r^* \in \cW_k^\circ$ and let $ \mu \in \ker H, \mu \gg 0$. Then there exists $t>0$ sufficiently large such that $(r^*+t \mu) \in \mathbb R_+^\nu$. Furthermore, $ H(r^*+t \mu) = Hr^*$, hence $ (r^*+t \mu) \in \cW_k^\circ$.\\
c) Assume that $r^* \in \cW_k^\circ$, then $-r^* \in -\cW_k^\circ$, which implies that the latter is nonempty and there exists $j \in \{1,..,m\}$ such that $\cW_j=-\cW_k$.
\hspace{\fill} $\blacksquare$

\paragraph*{Proof of Theorem \ref{th.lyap}}

Take the Dini derivative along solutions of (\ref{e.ode}) to get
\begin{align*}
\label{dini1} \nonumber
 D^+ V(x(t))  &:= \limsup_{h \rightarrow 0^+} \frac{V(x(t+h))- V(x(t))}{h} \\ & =   \limsup_{h \rightarrow 0^+}  \frac{\tilde V ( R(x(t+h)) )- \tilde V (R(x(t)))}{h} \\
&\mathop{=}^{(\star)} \limsup_{h \rightarrow 0^+}  \frac{\tilde V ( R( x(t) )+h \dot{R}(x(t)) )  ) - \tilde V (R(x(t)))}{h} \\ & \mathop{\le}^{(*)} \max_{k \in K_{x(t)}} c^{T}_k \dot{R} (x(t)) = \dot{V} (x (t) )     ,
\end{align*}
where $\dot R(x)= \frac{\partial R(x)}{\partial x} \Gamma R(x) $, and the equality $(\star)$ follows from Taylor's expansion of $R(x(t))$ with respect to time 
and considering that $\tilde V(r)$ is a Lipschitz function. Furthermore, by Danskin's Theorem \cite{clarke97}, $(*)$ will be an equality if $\tilde V$ was convex.   \\
Therefore, the claims of Theorem \ref{th.lyap} follow from the Lyapunov Second's Method and Krasovskii-LaSalle's principle \cite{yoshizawa} since by assumption, the solution $x(t)$ is bounded, so the $\omega$-limit set $\omega(x(0))$ is non-empty and compact, hence a subset of $\ker \dot V$. \strut\hspace{\fill} $\blacksquare$

\paragraph*{Proof of Corollary \ref{cor1}}
Since $x^\star$ is isolated, then it is asymptotically stable as well as $E_{x_\circ} \backslash \{x^\star\}$. Let $\mathcal B_{x^*}$ be its basin of attraction, and let $\mathcal B_c$ be the basin of attraction of $E_{x_\circ} \backslash \{x^\star\}$. The standard arguments can be replicated to show that $\mathcal B_{x^*}$ and $E_{x_\circ} \backslash \{x^\star\}$ are open relative to $\mathscr C_{x_\circ}$ \cite{hahn67}. However, as all trajectories starting in $\mathscr C_{x_\circ}^\circ$ converge to the equilibrium set, this implies that $\mathscr C_{x_\circ}= \mathcal B_{x^*} \cup \mathcal B_c$. Hence, the connected open set $\mathscr C_{x_0} \cap \mathbb R_+^n$ is a union of two disjoint open sets; a contradiction. Therefore, $E_{x_\circ}=\{x^*\}$ and $\mathcal B_{x^*}=\mathscr C_{x_\circ}^\circ$. \strut\hspace{\fill} $\blacksquare$


\paragraph*{Proof of Theorem \ref{th.checkcont}:} We show that each condition is equivalent to a required property for the Lyapunov function:

\emph{C1: Nonnegativity:} The inequality $V(x)\ge 0$ holds if and only if $c_k^T r\ge 0$ whenever $\Sigma_k H r\ge 0, r\ge 0$. By the Farkas Lemma \cite{rockafellar}, this is equivalent to the existence of $\xi_k,e_k \in \mathbb R^\nu, [\xi_k^T,e_k^T] > 0, k=1,..,m/2$ so that $c_k = \xi_k ^T \Sigma_k H + e_k^T.$
We show that $e_k=0$. Note that by symmetry we have $-c_k=c_{-k}= -\xi_{-k} ^T \Sigma_k H + e_{-k}^T$ with $[\xi_{-k}^T,e_{-k}^T] > 0$. Since there exists $v \in \ker \Gamma=\ker H$ with $v \gg 0$, we have $c_k^T v = e_{k}^T v \ge 0$, and     $c_k^T v =  -e_{-k}^T v \le 0$. This implies $e_k=0$.

\emph{C2: Positive-Definiteness:} Let $R(x) \in \ker C$ be arbitrary, we see that $V(x) =0$ and therefore, by definition of PWLR Lyapunov function, $R(x) \in \ker \Gamma$. Thus, $\ker C \subset \ker \Gamma$. \\
 To show the converse direction, note that C1 implies that $\ker H \subset \ker C$. However, we assumed that $\ker H=\ker \Gamma$ and we have shown in above that $ \ker C \subset \ker \Gamma$ . Hence, $\ker C=\ker \Gamma$. Hence, the statement ``$V(x)=0$ iff $x$ is an equilibrium'' is equivalent to $\ker C=\ker \Gamma$.

\emph{C3: Continuity:} 
Suppose $\tilde V$ is continuous, and let $(k,j) \in \mathcal N$, i.e. $\cW_k,\cW_j$ are neighboring regions. Then $c_k^T r= c_j^T r$ whenever $r \in \partial \cW_k \cap \partial \cW_j=\{r | h_{s_{kj}}^T r=0\}$. Hence, $(c_k-c_j) \in \im(h_{s_{kj}})$, which implies \eqref{e.continuety}.

Assume now that the converse is true. We need to show that $c_k^T r= c_j^T r$ whenever $r \in \partial \cW_k \cap \partial \cW_j$. The statement is true when $d_r(\cW_k,\cW_j)=1$ by \eqref{e.continuety}. Thus, we show it when $d_r(\cW_k,\cW_j)>1$. We need to introduce the following lemma:
\begin{lemma} \label{lem.cont_induction}
 Let $\tilde H$, and $\{\cW_k\}_{k=1}^m$ be as above. Assume that $d_r(\cW_k,\cW_j)=N>1$, then $\exists \cW_\ell, 1\le \ell \le m, \ell \ne k,j$ such that $d_r(\cW_k,\cW_\ell)+d_r(\cW_\ell,\cW_j)=N$.
\end{lemma}
\begin{proof} We use mathematical induction. Assume that $\tilde N=2$. We can represent $\cW_j, \cW_k$, w.l.o.g, by the matrices $H_k=[(\hat \Sigma_k \hat H)^T \, h_{t_1} \, h_{t_2} ]^T, H_j=[ (\hat \Sigma_k \hat H)^T \, -\!h_{t_1} \, -\!h_{t_2} ]^T$. For the sake of contradiction, assume there does not exist $\cW_\ell$ satisfying the claim. This implies that the regions represented by the matrices $[(\hat \Sigma_k \hat H)^T \, h_{t_1} \, -\!h_{t_2} ]^T, [ (\hat \Sigma_k \hat H)^T \, -\!h_{t_1} \, h_{t_2} ]^T$ have empty interiors. 
By Farkas Lemma, there exists $\xi_1,\xi_2 \in  \bar{\mathbb R}_+^{\tilde p-2}, \xi_{1t_1}, \xi_{2t_2} \in \bar{\mathbb  R}_+ $ with: $ h_{t_1} = \xi_1 \hat H + \xi_{1t_1} h_{t_2} = -(\xi_2 \hat H -\xi_{1t_2} h_{t_2}). $ 
Hence, $(\xi_1+\xi_2) \hat H = ( \xi_{1t_1} - \xi_{1t_2}) h_2 $.  If $ \xi_{1t_1} \ne \xi_{1t_2}$, then either $[\hat H \, -\!h_{t_2}]$ or $[\hat H \, h_{t_2}]$ correspond to a region with empty interior, which is not. Thus, $ \xi_{1t_1} = \xi_{1t_2}$, which implies $(\xi_1+\xi_2) \hat H = 0 $. Since $\cW_j^\circ, \cW_k^\circ \ne \varnothing$, then $\xi_1=\xi_2=0$. Thus, we get $ h_{t_1} = \xi_{1t_1} h_{t_2}$, which contradicts our assumption that $\tilde H$ does not have linearly dependent row pairs. Therefore, the statement is true for $N=2$.\\
Assume now that the statement is true for $\tilde N=N-1$, and let $\tilde N=  N$. We can represent $\cW_j, \cW_k$, w.l.o.g, by the matrices $H_k=[(\hat \Sigma_k \hat H)^T \, h_{t_1} \, ..\, h_{t_N} ]^T, H_j=[ (\hat \Sigma_k \hat H)^T \, ..\, -\!h_{t_1} \, -\!h_{t_N} ]^T$. Let $\cW_k^-, \cW_j^-$ be the regions corresponding to the matrices $H_k^-=[(\hat \Sigma_k \hat H)^T \, h_{t_1} \, ..\, h_{t_{N-1}} ]^T, H_j^-=[ (\hat \Sigma_k \hat H)^T \, ..\, -\!h_{t_1} \, -\!h_{t_{N-1}} ]^T$. By the induction hypothesis, $\exists \cW_\ell^-$ satisfying the claim. Moreover, either $ \cW_\ell^- \cap \{r|h_N r\ge 0 \}$,$\cW_\ell^- \cap \{r|h_N r \le 0 \}$ or both have a nonempty interior. Let $\cW_\ell$ be a nonempty one. Thus, $\cW_\ell$ satisfies the claim. \end{proof}

We are ready to prove continuity now. We can write $\partial \cW_k \cap \partial \cW_\ell = \{r| [h_{t_1}^T \, .. \, h_{t_N}^T] r\ge 0\}$. By Lemma \ref{lem.cont_induction}, it can been seen that for any $\cW_{k_1}, \cW_{k_N}$ with $d(\cW_{k_1},\cW_{k_N})=N$ there exists a sequence of regions $\cW_{k_2}, .. , \cW_{k_{N-1}}$ so that $d_r(\cW_{k_{\ell}}, \cW_{k_{\ell+1}} ) =1 $ and $s_{k_\ell}(k_{\ell+1})=h_{t_\ell}$. Hence:\vspace{0em}
\[ c_{k_N}-c_{k_1}= \sum_{\ell=1}^N (c_{k_{\ell+1}}-c_{k_\ell}) = \sum_{\ell=1}^N \eta_{k_{\ell+1} k_{\ell}} h_{t_\ell},\vspace{0em} \]
which implies $(c_{k_N}-c_{k_1})^Tr = 0 $ when $r \in \partial \cW_{k_1} \cap \mathcal \partial W_{k_N}$.\\
\emph{C4: Nonincreasingness:}
When $R(x) \in \cW_k$, we can write:\vspace{0em}
\begin{align}\label{e.Vdot_exp}
 \textstyle\dot V(x) &= c_k^T \dot R(x) = c_k^T \frac{\partial{R}}{\partial x}(x) \Gamma R(x)   \\ \nonumber & 
 = \sum_{i \in I_k}  \sum_{j \in J_{ki}} c_{kj}  \frac{\partial R_j}{\partial x_i} \dot x_i \le 0.\vspace{0em}
\end{align}
We claim that this is equivalent to the statement ``$c_{kj} \dot x_i \le 0$ whenever $R(x) \in \cW_k$, for all $j \in J_{ki}, i \in I_k, k=1,..,m/2$''. Since the sufficiency is clear, we just show necessity:  assume that there exists $j^* \in J_{ki}, i^* \in I_{k}, 1 \le k^* \le m/2$ so that $c_{k^*j^*} \dot x_i^* \ge 0$. Then, we can choose $({\partial R_{j^*}}/{\partial x_{i^*}})(x)$ large enough so that the corresponding system in the network family $\mathscr N_\Gamma$ will have $\dot V(x) \ge 0$.

Now we show equivalence with conditions a)-b). Considering the statement above and since $R$ is monotone, this entails that  $\sgn(c_{kj_1})\sgn(c_{kj_2})\ge 0$ for every $j_1,j_2 \in J_{ki}$, which shows condition a). Thus, we define $\nu_{ki}=\sgn(c_{kj^*}), j^* \in J_{ki}$. To show b), By Farkas Lemma, the condition is equivalent to the existence of $\lambda^{(ki)} \in \bar{\mathbb R}_+^n$, $i \in I_k, k=1,..,m/2$ so that \eqref{e.decreasing} holds. \\
 It remains to be shown that we can choose the coefficients so that $\supp(\lambda^{(ki)}) \subset s_k(\mathcal N_k)$. This follows directly from the following lemma:
\begin{lemma}\label{lem.basis}
 Let $H$, and $\{\cW\}_{k=1}^m$ be as above, then the rows of $H_k=[\sigma_{ks_{k\ell_1}} h_{s_{k\ell_1}} \ .. \ \sigma_{ks_{k\ell_{o_k}}} h_{s_{k\ell_{o_k}}} ]^T$ form a conic basis of $\mathcal W_k$, where $o_k=|\mathcal N_k|$
\end{lemma}
\begin{proof}
 By definition, each row vector of $H_k$ is conically independent of the rows of $\Sigma_k H$. It remains to be shown the rows of $\Sigma_k H$ belong to the conic span of the rows of $H_k$. Assume, w.l.o.g, that $\Sigma_k H=[\hat H^T \, h_{t_1} \, h_{t_2} ]^T$, where $h_{t_1},h_{t_2}$ are not conically independent, and we need to show that they belong to the conic span the rows of $\hat H$. Then by Farkas Lemma $\exists \xi_1,\xi_2 \in  \bar{\mathbb R}_+^{  p-2}, \xi_{1t_1}, \xi_{2t_2} \in \bar{\mathbb  R}_+ $ with: $ h_{t_1} = \xi_1 \hat H + \xi_{1t_1} h_{t_2}, h_2= \xi_2 \hat H +\xi_{2t_2} h_{t_2}$. By substitution, we get $(1-\xi_{1t_1}\xi_{1t_1}) h_2 = (\xi_2+\xi_{2t_2})\hat H$. If $1-\xi_{1t_1}\xi_{1t_1}<0$, then this contradicts that $\cW_k$ has nonempty interior. If $1-\xi_{1t_1}\xi_{1t_1}=0$, then this contradicts $v \in \ker \hat H$ for $v\gg 0$. Therefore, the only possibility is $1-\xi_{1t_1}\xi_{1t_1}>0$, which proves the claim for two vectors. In general, this procedure can be applied to eliminate all conically dependent rows. 
\end{proof}


\paragraph*{Proof of Theorem \ref{th.checkcvx}} The converse direction of C2$'$ can be shown directly since if $c_k^T R=0$, convexity implies that $c_k^TR=0, k=1,..,m$. Hence $\ker \Gamma \subset \ker C$. C4$'$ can be shown via a similar argument to the previous proof.

\paragraph*{Proof of Theorem \ref{cont.vs.cvx}} Let $V(x)=\tV(R(x))$ be a PWLR function, and denote its polyhedral level set by $G=\{r \in \mathbb R^\nu | \tV(r) \le 1 \}$. By homogeneity and nonnegativity of $\tV$ we can write\vspace{0em}
\begin{equation}\label{e.minkowski}
\tV(r) = \inf_{r \in c G } c. \vspace{0em}\end{equation}
 Note that $\tV$ will be a Minkowski functional if $G$ was convex. As the level set $G$ characterizes $\tV$ fully, we want to express C4 for the set $G$. To that end, we use the notion of a \emph{tangent cone}, which we define as follows for a polyhedral set $G$ induced by a PWL function $\tV$ at point $r$: $T_r G := \bigcap_{k \in K_r} \{ z \in \mathbb R^\nu| c_k^T z \le 0 \} $, where $K_r= \{ k \in \{1,..,m\} | r \in \cW_k \}$. In fact, our definition of $T_r G$ coincides with \emph{Clarke's Tangent Cone} \cite{clarke97}. We state now the following Lemma:
\begin{lemma}\label{lem.tan} Given a polyhedral set $G \subset \mathbb R^\nu$. 
 The induced PWLR function $V=\tV(R(x))$ with $\tV$ as in \eqref{e.minkowski} satisfies C4 if and only if $ \frac{\partial R}{\partial x}(x)\Gamma R(x) \in T_{R(x)} G $ for all $R\in \mathscr K_\Gamma$.
\end{lemma}
\begin{proof} Note that the condition $ \frac{\partial R}{\partial x}(x)\Gamma R(x) \in T_{R(x)} G $ is equivalent to the requirement that $c_k^T  \frac{\partial R}{\partial x}(x)\Gamma R(x) \le 0$ for all $k \in K_{R(x)}$, and all Jacobian matrices corresponding to reaction rates in $\mathscr K_\Gamma$ . This is equivalent to C4 as can be noted in the proof of Theorem \ref{th.checkcont} and \eqref{e.Vdot_exp}.
\end{proof}

Consider a possibly nonconvex PWLR function $V$, and let $G$ be defined as above. We need the following lemma to proceed:
\begin{lemma}\label{lem.cone_addition} Let $r_1, r_2 \in G, \alpha \in [0,1]$. Denote $r=\alpha r_1 + (1-\alpha )r_2 \in \co(G)$. Then, $T_{r_1} G \oplus T_{r_2} G \subset T_{r} \co(G)$, where $\oplus$ denotes conic addition of sets.
\end{lemma}
\begin{proof} Let $z_1 \in T_{r_1} $, $z_2 \in T_{r_2}$, and $h_n \ssearrow 0$. By the definition of Clarke's tangent cone \cite{clarke97}, there exists $z_{1n} \to z_1, z_{2n}\to z_2$ such that $r_1+h_nz_{1n}, r_2+h_nz_{2n} \in G$. Let $z=\alpha z_1 + (1-\alpha)z_2$, and $z_n=\alpha z_{1n} + (1-\alpha) z_{2n} \to z $. Then, we have $r+h_n z_n \in \co(G)$. Let $r_n=r+h_nz_n \to r$. Thus, $ \frac 1{h_n}(r_n- r) \to z$. Hence, $z \in T_r^{(B)}\co(G)$, where $B$ denotes the Bouligand's tangent cone \cite{clarke97}. However, as the two cones are identical for convex sets, then $ z \in T_r \co(G)$.  The argument can be applied to any nonnegative combination with appropriate scaling of $h_{1n}, h_{2n}$.
\end{proof}

Now, let $R(x) \in \co(G)$. Hence, there exist $x_1,x_2$ and $\alpha\in[0,1]$ such that $R(x) = \alpha R(x_1) + (1-\alpha) R(x_2)$. Therefore, we can write:\vspace{0em}
\begin{equation}\label{e.hull} \textstyle \frac{\partial R}{\partial x} (x)\Gamma R(x) = \alpha \frac{\partial R}{\partial x} (x) \Gamma R(x_1) + (1-\alpha) \frac{\partial R}{\partial x} (x)  \Gamma R(x_2).\vspace{0em} \end{equation}
 By Lemma \ref{lem.tan}, $\frac{\partial R}{\partial x} (x) \Gamma R(x_1) \in T_{R(x_1)} G$, $\frac{\partial R}{\partial x} (x) \Gamma R(x_2) \in T_{R(x_2)}G$ for all $x$. Therefore, Lemma \ref{lem.cone_addition} implies that $\frac{\partial R}{\partial x} (x)\Gamma R(x) \in T_{R(x)} \co(G)$ for all $x$. Therefore, $ V(x) = \inf_{R(x) \in r \mbox{\footnotesize co}(G)} r$, is a convex PWLR function. \hspace{\fill} $\blacksquare$

 \paragraph*{Proof of Proposition \ref{th.lasalle}}
Let $\tilde{x}(t)$ denote any solution of \eqref{e.ode} which is contained in $\ker(\dot V(x))$. Consider first the case when $\tilde x(0) \in \mathbb R_n^+$. Let $\mathcal T_k=\{t>0 : c_k^TR(\tilde x(t))=V(\tilde x(t))\}$, $k=1,..,m$, then $\mathcal T_k$ are closed relative to some maximally defined interval where the solution $\tilde{x}$ exists, and $\bigcup_k \mathcal T_k=(0,\tau_{x_\circ})$. The existence of an open set $\mathcal T$ and $k^\star$
 such that
$V(\tilde x(t)) = c_{k^\star}^T R(\tilde x(t))$ for all $t \in \mathcal{T}$  follows by the Baire Category Theorem \cite{royden88} . By C4, $c_{k^\star}^T \dot R(\tilde x(t))=0$ identically for $t \in \mathcal{T}$ implies $\dot R_j(\tilde x (t) ) = 0$, for all $t \in \mathcal{T}$ and $j \in J_{k^\star}$. Then, by A4, we have  $\dot{ \tilde{x}}_i (t) =0$ identically for $t \in \mathcal{T}$ and all $i \in I_{k^\star}$. Using \eqref{e.decreasing}, $h_i^T R(\tilde x (t))=0$ for $t \in \mathcal T$ and all $i \in \supp(\lambda^{({k^\star}i)})$. By \eqref{e.continuety}, $c_j^T R(\tilde x(t))= c_{k^\star}^T R(\tilde x(t))$ and $\tilde{x}(t) \in \mathcal{W}_j$ for all $j \in L_{{k^\star}i}$ and all $t \in \mathcal{T}$. Hence, $c_j^T \dot R(\tilde x(t))=0, j \in L_{k^\star}$.  Iterating this procedure, we get $c_j^T \dot{R} ( \tilde x (t)) =0$ for all $j \in L_k^{( i^\star)}$ and accordingly, $\dot {\tilde x}_i(t)=0$ for all $i \in \bar I_{k^\star}=\{1,..,n\}$. Hence $\tilde{x}(t)$ is a constant solution and belongs
to the set of equilibria.
Additional comment is needed for C5$'$i: If $c_k \in \im(\Gamma_{\bar I_k}^T)$, then $c_k^T R(\tilde x(t)) = 0$. By convexity, this implies that $V(\tilde x(t))=0$, and hence $\tilde{x}(t) \in E$.

Assume now that $\tilde x(0)$ belongs to non-invariant face of $\mathscr C_{x_\circ}$, then $\tilde x(t) \in \mathbb R_+^n$ for $t>0$ and hence the argument of the previous case still applies. Finally, if $\tilde x(0)$ belongs to a closed invariant face $\Psi_P$ we may regard the solution $\tilde{x} (t)$ as a solution of the subnetwork obtained by deleting all species that are zeroed in $\Psi_P$ and removing all their associated output reactions.

Then, with a recursive argument, three cases arise, either $\tilde{x}(0)$ belongs to the interior of the stoichiometry class associated to the subnetwork, or it belongs to one of its non-invariant face or it belongs to an invariant face.
Since we assumed that C5i applies to each critical subnetwork and in turn subnetworks of subnetworks are regarded as subnetworks themselves, we can continue this recursive procedure to show that  $\tilde{x}(t) \in E_{x_\circ}$ for any initial condition $\tilde x(0) \in  \mathscr C_{x_\circ}$. \strut \hfill $\blacksquare$

\paragraph*{Proof of Theorem \ref{th.p0}} Consider the ODE \eqref{e.ode}. Assume that there exists a PWLR Lyapunov function. As explained in \S \ref{sec.LPalg}, the PWL function $\tV$ can be considered to be defined over a partition generated by a matrix of the form $\hat H=[\Gamma^T \, H^T]^T$, with $\{\cW_k\}_{k=1}^m$. \\ Fix $k \in \{1,..,m\}$, and let $c_k$ be the corresponding coefficient vector. By C1, we have $c_k^T = \xi_k^T \Sigma_k \hat H=\xi^T \diag(\Sigma_k^{(s)} ,\Sigma_k^{(h)})   [\Gamma^T \, H^T]^T  $ with $\xi_k >0$. By \eqref{e.Vdot_exp}, the following hold over $\cW_k$:\vspace{0em}
\begin{equation*}
  \xi_k^T \Sigma_k [\Gamma^T \, H^T]^T \frac{\partial R}{\partial x}(x) \dot x  \le 0.\vspace{0em}
\end{equation*}
However,  $R(x) \in \cW_k$ implies that $\dot x \in \mathcal S_{q_k} $, i.e. the sign pattern of $\dot x$ is identical to $\Sigma_k^{(s)}$. Furthermore, as noted in the proof of condition C4 of Theorem \ref{th.checkcont}, every term in the expansion of left-side of the above inequality is nonpositive. Hence, the following holds over the region $\cW_k$ for all $x$: \vspace{0em}
\begin{align}\label{e.p0_proof}
  &\xi_k^T \Sigma_k \begin{bmatrix} \Gamma \\ H \end{bmatrix} \frac{\partial R}{\partial x}(x) \Sigma_k^{(s)} \\ \nonumber &
  =   \xi_k^T \begin{bmatrix} \Sigma_k^{(s)} & 0 \\ 0 & \Sigma_k^{(h)} \end{bmatrix} \overbrace{\begin{bmatrix} \Gamma    \frac{\partial R}{\partial x}(x) & 0\\ H   \frac{\partial R}{\partial x}(x) & 0 \end{bmatrix}}^{J}  \begin{bmatrix} \Sigma_k^{(s)} & 0 \\ 0 & \Sigma_k^{(h)} \end{bmatrix} \le 0. \vspace{0em}
\end{align}
Now, consider the case $k \in \{m+1,..,2^p\}$. By definition, $\cW_k^\circ=\varnothing$. By Farkas Lemma, there exist $t \in \{1,..,p\}, \xi_k>0$ with $\xi_{kt}=1$ such that $\sigma_{kt} h_t=-\sum_{i \ne t } \xi_{ki} \sigma_{ki} h_i $. Therefore, $\xi_k^T \Sigma_k[ \Gamma^T H^T ] ^T = 0$. Hence, the inequality \eqref{e.p0_proof} is satisfied with equality.
 Therefore, we have shown that for all signature matrices $\{\Sigma_k\}_{k=1}^{2^p}$, there exists $\xi_k >0$ such that $-\xi_k^T \Sigma_k J \Sigma_k \ge 0$. Indeed, this is a characterization of $P_0$ matrices \cite[p. 149]{berman94}. Hence, $-J$ is a $P_0$ matrix for all $x$. In particular, this implies that $-\Gamma \frac{\partial R}{\partial x}(x)$ is $P_0$ for all $x$. \hfill $\blacksquare$

\paragraph*{Proof of Theorem \ref{th.necessary}}
Assume that there exists a pair $C \in \mathbb R^{m_h/2 \times \nu},H \in \mathbb R^{p \times \nu}$ such that $\tilde V$ \eqref{e.conLF} is PWLR Lyapunov function, and let $\{B_k\}_{k=1}^{m}$ be as defined in Theorem.  Since $\{\mathcal S_k\}_{k=1}^m$ is a partition, then $\forall k \in \{1,..,m\}, \exists \ell \in \{1,..,m_h\}$ such that $\cW_\ell^{\circ} \cap \mathcal S_k^\circ \ne \varnothing$. As in the proof of Theorem \ref{th.checkcont}, C4 is equivalent to requiring $c_{\ell j} \dot x_i \le 0$ whenever $R(x) \in \cW_\ell$, for all $j \in J_{\ell i}, i \in I_\ell, \ell=1,..,m_h/2$. Hence,
$ c_{\ell j} = | c_{\ell j} | b_{k j}, j \in \{1,..,\nu\}\ \mathcal I. $
Therefore, $\exists \zeta_k \in \mathbb R^\nu$ such that $c_k=\zeta_k^T B_k$, with $\zeta_k \ne 0, \zeta_{kj}\ge 0, j \in \{1,..,\nu\}\ \mathcal I$. Furthermore, since $\ker C=\ker \Gamma$, then $\zeta_k^T B_k V =0 $. \strut \hfill $\blacksquare$

\paragraph*{Proof of Corollary \ref{cor.criticalSiphon}:} Without loss of generality, let $\{1,...,t\}$ be the indices of the species in $P$. We claim that this implies that there exists a nonempty-interior sign region $\mathcal S_k, 1 \le k \le m_s$ with a signature matrix $\Sigma_k$ that satisfies $\sigma_{k1}=...=\sigma_{kt}=1$. To prove the claim, assume the contrary. This implies that $ \cap_{i=1}^{t} \{ R | \gamma_{i } ^ T R > 0 \} \, \bigcap \, \cap_{i=t+1}^{n} \{ R | \sigma_{i} \gamma_{i } ^ T R > 0 \} = \emptyset$ for all possible choices of signs $\sigma_i=\pm 1$. However, $\mathbb R^r$ can be partitioned into a union of all possible half-spaces of the form $\cap_{i=t+1}^{n} \{ R | \sigma_{i} \gamma_{i } ^ T R \ge 0 \}$. Therefore, this implies that $ \cap_{i=1}^{t} \{ R | \gamma_{i } ^ T R > 0 \} = \emptyset$. By Farkas Lemma, this implies that there exists $\lambda \in \mathbb R^t$ satisfying $\lambda > 0$ such that $ [\lambda^T 0] \Gamma =0$. Therefore, $P$ contains the support of the conservation law $[\lambda^T \,  0]^T$; a contradiction. \\
Now consider $\mathcal S_k$ with $\sigma_{k1}=...=\sigma_{kt}=1$. Since $\Lambda(P)={\rm \bf V}_R$, this implies $b_{kj}\le 0$ for all $j=1,..,\nu$. However, this is not allowable by Theorem \ref{th.necessary} since $ \zeta_k B_k v \le 0 $ for all $v \in \ker \Gamma \cap \Rnn$ and for any choice of admissible $\zeta_k$. \strut\hfill $\blacksquare$

\paragraph*{Proof of Theorem \ref{th.fixedH}}
Note that C1,C3 are represented by first and third constraints in the linear program. It remains to be shown that C4 is equivalent to the second constraint. In fact, using the same argument in the proof of Theorem \ref{th.necessary}, this is equivalent to the coefficients $c_k$ being compatible with the sign region $S_{q_k}$, which is equivalent to the second constraint. \strut \hfill $\blacksquare$

\paragraph*{Proof of Theorem \ref{th_alg}}
\begin{figure*}
\hrulefill
\begin{align}\label{wideeq}\nonumber\dot{\tilde R}_q(x) & = \frac 1{v_{q^\star}}\sum_{i: \alpha_{i q^\star } >0} \frac{\partial R_{q^\star}}{\partial x_i}  \left ( - \alpha_{ i q^\star  } (R_{q^\star}(x)-R_{-q^\star}(x)) + \sum_{\ell \ne q^\star} \beta_{ i\ell} (R_{\ell}(x) -R_{-\ell}(x) ) \right ) \\ & \quad - \frac 1{v_{q^\star}}\sum_{i: \beta_{ i q^\star } >0} \frac{\partial R_{-q^\star}}{\partial x_i}  \left ( - \beta_{i q^\star  } (R_{-q^\star}(x)-R_{q^\star}(x)) + \sum_{\ell\ne q^\star} \alpha_{ i\ell} (R_{-\ell}(x) -R_{\ell}(x)) \right ),\vspace{0em}\end{align}
\hrulefill
\end{figure*}

We need to show that $C$ satisfies C4$'$. Thus, it is sufficient to show that $\cW_1 (c_k) \subset \mathcal P(c_k)$, note that $R \in \cW_1(c_k)$ implies $( c_k + \nu_{ki}\gamma_i)^T R  \le c_k^T R $, for $i=1,...,n$. This in turn implies $\nu_{ki}\gamma_i^T R  \le 0, i=1,..,n$. Hence, $R \in \mathcal P(c_k)$.
 Therefore, by iterating this procedure through the rows of $C_m$ we get $ \cW_{k} (c_k) \subset  ... \subset   \cW_1 (c_k) \subset \mathcal P(c_k) $. If the algorithm terminates after finite number of iterations then a nonincreasing convex PWL function is constructed. Furthermore, to ensure that $\ker C \subset \ker \Gamma$ then we assume that $\ker \Gamma \subset \ker C_0$. \strut \hfill $\blacksquare$
\paragraph*{Proof of Theorem \ref{th.maxmin}}
  The PWLR function is convex by construction, where $C \in \mathbb R^{ \frac \nu 2(\nu-1) \times \nu}$. We can write $CR(x) = [c_{q+s}^T R(x)]_{s,q=1}^{q-1,\nu} = [\frac 1{v_q}R_q(x) - \frac 1{v_s}R_s(x)]_{s,q=1}^{q-1,\nu}$. Note that $\tV(R(x))= c_{q+s}^T R(x)$ is equivalent to  $\max_{1\le k \le \nu} \frac 1{v_k} R_k(x) = R_q(x) $ and $\min_{1\le k \le \nu} \frac 1{v_k} R_k(x) = R_s(x) $. \\
  Using the first assumption, we have $\ker \Gamma = \ker C$, and hence C2$'$ is satisfied. \\
We show C4$'$ using a directly. By assumption 2, we can perform the following computation:
\begin{align} \label{e.proof1}
\dot R_q (x) 
& = \frac 1{v_{q}}\frac{\partial R_{q} }{\partial x} \, \Gamma R(x) 
       \\ \nonumber & = \frac 1{v_{q}}\sum_{i: \alpha_{ i q} >0} \frac{\partial R_{q}}{\partial x_i}  \left ( - \alpha_{ iq} R_{q}(x) + \sum_{j\ne q} \beta_{i j} R_{j}(x) \right ) \end{align} 
       \begin{align*} \dot R_s (x) &= \frac 1{v_{s}}\frac{\partial R_{s} }{\partial x} \, \Gamma R(x)  
    \\ \nonumber &  =  \frac 1{v_{s}}\sum_{i: \alpha_{ is  } >0} \frac{\partial R_{s}}{\partial x_i}  \left ( - \alpha_{is } R_{s}(x) + \sum_{j \ne s} \beta_{i j} R_{j}(x) \right )
\end{align*}
Since $v \in \ker \Gamma$, then $-\alpha_{iq } = \sum_{j \ne q} \frac{v_{q}}{v_j} \beta_{ i j }$. Hence,
\begin{align*}& - \alpha_{i q  } R_{q}(x) + \sum_{j\ne q} \beta_{i j } R_{i}(x) \\ &\leq -\sum_{j \ne q} \frac{v_{q}}{v_j} \beta_{ j i} R_q(x) +  \sum_{j\ne q} \beta_{ j i } \frac{v_q}{v_{j }}R_{q}(x)= 0.\end{align*}
By a similar argument for the second term and by \eqref{e.rate_mono}, we get $\dot R_q(x) \le 0 $ and $\dot R_s(x) \ge 0$ and as a consequence $\dot V(x) = \max_{k \in K_x} c_k^T \dot R(x) \leq 0$. Hence, $V$ is a PWLR Lyapunov function.

We show the LaSalle's interior condition. Let $\tilde x (t) \subset \ker \dot V(x)$. Using the same argument in the proof of Theorem \ref{th.lasalle}, there exists $q^{\star}$ and $s^{\star}$ and an open subset of $\mathbb{R}$, $\mathcal{T}$, such that
$V(x)=\frac 1{v_{q^\star}} R_{q^\star} ( \tilde{x} (t) )- \frac 1{v_{s^\star}} R_{j^{\star}} ( \tilde{x} (t) )$ for all $t \in \mathcal{T}$. However, as both terms are nonincreasing, we have $\dot R_{q^\star}(\tilde{x}(t))=\dot R_{s^\star}( \tilde{x}(t) )=0$. By \eqref{e.proof1} and assumption 4 in \S2 we get $\dot{ \tilde{x}}_i (t) =0$ for all $i$ such that $  \alpha_{i q^{\star} } >0$, $\alpha_{i s^{\star}  } >0$ and all $t \in \mathcal{T}$. Therefore,
$  \alpha_{i q^\star  } R_{q^\star}(\tilde{x}) =  \sum_{j\ne q^\star} \beta_{i j } R_{j}(\tilde{x})$, and since $\frac 1{v_j} R_j(\tilde{x}(t)) \le \frac 1{v_{q^\star} } R_{q^\star}(\tilde{x}(t)) $, then $\frac 1{v_{q^\star}}R_j(\tilde{x} (t) )=R_{q^\star}(\tilde{x}(t))$ for all $j$ such that there is $i$ with $\beta_{i j}>0$, $\alpha_{ i q^\star }>0$ and all $t \in \mathcal{T}$.
A similar argument can be carried out for $R_s(\tilde{x}(t))$. By induction, it follows that $\frac 1{v_{q^\star}} R_j(\tilde{x}(t))=R_t(\tilde{x}(t))$ if $\R_j \in \mathscr A (\R_{q^\star})$ for $t \in \mathcal{T}$. Similarly, $\frac 1{v_{q^\star}} R_\ell(\tilde{x}(t))=R_s( \tilde{x}(t))$ if $\R_\ell \in \mathscr A (\R_{s^\star})$. Since $ \mathscr A (\R_{q^\star}) \cap  \mathscr A (\R_{s^\star}) \ne \phi $, we get $R_{q^\star}(\tilde{x}(t))= R_{s^\star}(\tilde{x}(t))$ for all $t \in \mathcal{T}$ and since $\mathcal{T}$ is an open set this implies that $\tilde{x}(t) \in \ker \Gamma$. \\
Assume the network is conservative. We claim that there are no critical siphons, and hence no critical subnetworks. For the sake of contradiction, assume that $P$ is a critical siphon. Hence, the associated face $\Psi_P$ is an invariant, compact and convex set. Applying the Brouwer fixed point theorem \cite{royden88} for the associated flow, there exists an equilibrium $x^* \in \Psi_P$ such that $\Gamma R(x^*)=0$. Since $\dim(\ker \Gamma)=1$, this implies that $R(x^*)=t v$ for some $t \ge 0$. Consider the case $t=0$. This implies $R(x^*)=0$. Then, $P \subset \tilde P:= \{1,..,n\}\backslash \supp(x^*)$\footnote{Equality is in the sense of the bijection between $\{1,..,n\}$ and $\mathbf V_S$.}. Observe that $\tilde P$ is a critical deadlock, however, this is not allowed by Corollary \ref{cor.criticalSiphon}. If $t>0$, this implies that $P=\emptyset$; a contradiction.

The persistence of the network follows directly from the absence of critical siphons by the results of \cite{Angeli07}. Since there exists a conservation law, the common ancestor condition is satisfied. Hence, the LaSalle's condition is satisfied. If there exists an isolated equilibrium, uniqueness and global stability follows for from Corollary \ref{cor1}. \hfill $\blacksquare$


\paragraph*{Proof of Theorem \ref{th2}:}
Without loss of generality, assume ${\bf V}_{R^{'}}={\bf V}_R$. As in the proof of Theorem \ref{th.maxmin}, we can write $CR(x) = [c_{q+s}^T R(x)]_{s,q=1}^{q-1,\nu} = [\frac 1{v_q}(R_q(x)-R_{-q}(x)) - \frac 1{v_s}(R_s(x) - R_{-s}(x) )]_{s,q=1}^{q-1,\nu}$. For simplicity, denote $\tilde R_{q}= R_q - R_{-q}$, $\tilde R_{s}= R_s - R_{-s}$. Hence, we can write \eqref{wideeq},
and an analogous expression can be written for $\dot{\tilde R}_s(x)$. Having a single negative coefficient in every bracket follows from the additional assumption in the statement of the theorem.
Therefore, using a similar argument to the one in the proof of Theorem \ref{th.maxmin} it can be seen that $\dot{\tilde R}_q(x) \le 0$, and $\dot{\tilde R}_q(x)  \ge 0 $. Hence, $\dot V(x(t)) \le 0$. \\
 A similar argument to the one in proof of Theorem \ref{th.maxmin} can be carried out to show the LaSalle's Interior condition. The absence of critical siphons follows from Theorem \ref{th.maxmin} and that critical siphons are not created by adding reverse reactions. \\ \strut\hfill $\blacksquare$ \vspace{-1em}

 \bibliographystyle{IEEEtran}

\begin{IEEEbiography}[{\includegraphics[width=1in,height=1.25in,clip,keepaspectratio]{ma_photo}}]{Muhammad Ali Al-Radhawi} received his
B.Sc. and M.Sc. degrees in Electrical Engineering from University of Sharjah, UAE,
in 2008 and 2011, respectively. He is currently pursuing his Ph.D. degree at the Department
of Electrical \& Electronic Engineering, Imperial College London. His research interests include stability analysis and control synthesis for reaction networks and networked systems.

\end{IEEEbiography}
\begin{IEEEbiography}[{\includegraphics[width=1in,height=1.25in,clip,keepaspectratio]{angeli_photo}}]{David Angeli}
graduated in Control Engineering from
the University of Florence (1996) and obtained in 2000
a Ph.D. degree from the same university. In 2008 he
joined Imperial College London, where he is currently
a Reader in Nonlinear Systems. He is also a part-time
Associate Professor at the University of Florence, and a Fellow of the IEEE. His
research interests include: stability of nonlinear systems,
Model Predictive Control and Chemical Reaction Networks
Theory.

\end{IEEEbiography}
\end{document}